\documentclass{amsart}
\usepackage[dvipdfmx]{graphicx}
\usepackage[dvipdfmx]{color}

\definecolor{RedOrange}{cmyk}{ 0, 0.77, 0.87, 0}
\definecolor{RoyalPurple}{cmyk}{ 0.84, 0.53, 0, 0}
\definecolor{YellowGreen}{cmyk}{ 0.44, 0, 0.74, 0}
\definecolor{Fuchsia}{cmyk}{ 0.47, 0.91, 0, 0.08}
\definecolor{Blue}{cmyk}{ 0.84, 0.53, 0, 0}
\definecolor{BlueViolet}{cmyk}{ 0.84, 0.53, 0, 0}
\definecolor{Black}{cmyk}{ 0.75, 0.68, 0.67, 0.9}
\usepackage{xcolor}
\usepackage{verbatim}
\usepackage{amsmath}
\usepackage{amssymb, mathrsfs}
\usepackage{amsbsy}
\usepackage{amscd}
\usepackage{amsthm}
\usepackage{bm}

\newcommand{\rmb}{{\rm B}^M_{i,z}}
\newcommand{\rma}{{\rm A}}
\newcommand{\rmax}{{\rm A}_M^{\gamma_{0,x}}}
\newcommand{\rmt}{{\rm T}}
\newcommand{\rmf}{{\rm F}}
\newcommand{\var}{{\rm Var}}
\newcommand{\w}{\omega}

\newcommand{\R}{\mathbb{R}}

\renewcommand{\P}{\mathbb{P}}
\renewcommand{\O}{\mathbb{O}}
\newcommand{\N}{\mathbb{N}}
\newcommand{\e}{\varepsilon}
\newcommand{\E}{\mathbb{E}}
\newcommand{\Z}{\mathbb{Z}}
\newcommand{\I}{\mathbb{I}}
\newcommand{\pp}{\mathbb{P}}

\newcommand{\kA}{\mathcal{A}}
\newcommand{\kB}{\mathcal{B}}

\newcommand{\kO}{\mathcal{O}}
\newcommand{\kP}{\mathcal{P}}

\newcommand{\kF}{\mathcal{F}}
\newcommand{\kG}{\mathcal{G}}

\newcommand{\kE}{\mathcal{E}}

\newcommand{\lin}{\left[\kern-0.15em\left[}
\newcommand{\rin} {\right]\kern-0.15em\right]}
\newcommand{\linf}{[\kern-0.15em [}
\newcommand{\rinf} {]\kern-0.15em ]}
\newcommand{\ilin}{\left]\kern-0.15em\left]}
\newcommand{\irin} {\right[\kern-0.15em\right[}

\newtheorem{lem}{Lemma}[section]

\newtheorem{prop}[lem]{Proposition}
\newtheorem{theo}[lem]{Theorem}

\newtheorem{cor}[lem]{Corollary}

\newtheorem*{ack}{Acknowledgments}

\usepackage{color}
\definecolor{lilas}{RGB}{182, 102, 210}

\numberwithin{equation}{section}
\begin{document}

\title{First passage time of the frog model has  a sublinear variance}

\author{Van Hao Can}\address{Van Hao Can, Research Institute for Mathematical Sciences, Kyoto University, 606--8502 Kyoto, Japan 
	\&	Institute of Mathematics, Vietnam Academy of Science and Technology, 18 Hoang Quoc Viet, 10307 Hanoi, Vietnam}
\author{Shuta Nakajima}\address{Shuta Nakajima, Research Institute for Mathematical Sciences, Kyoto University, 606--8502 Kyoto, Japan }

	\maketitle
	
	\begin{abstract}
		In this paper, we show that the first passage time in the frog model on $\Z^d$ with $d\geq 2$ has a sublinear variance. This implies that the central limit theorem does not holds at least with the standard diffusive scaling. The proof is based on the method introduced in \cite{BRo, DHS} combining with a control of the maximal weight of paths in locally dependent site-percolation.  We also apply this method to get the linearity of the lengths of optimal paths.
	\end{abstract}
\section{Introduction}
Frog models are simple but well-known models in the study of the spread of infection. In these models, individuals (also called frogs) move on the integer lattice $\Z^d$, which have one of two states infected (active) and healthy (passive). We assume that at the beginning, there is only one infected frog at the origin, and there are healthy frogs at other sites of $\Z^d$. When a healthy frog encounters with an infected one, it becomes infected forever. While the healthy frogs do not move, the frogs  perform independent simple random walks once they get infected. We are interested in the long time behavior of the infected individuals.\\

To the best of our knowledge, the first result on frog models is due to Tecls and Wormald \cite{TW}, where they proved the recurrence of the model (more precisely, they showed that the origin is visited infinitely often a.s. by infected frogs). Since then, there are numerous results on the behavior of the model under various settings of initial configurations, mechanism of walks, or underlying graphs, see \cite{AMP,BDHJ,BR,DP,GS,HJJ,HJJ1,KZ}. In particular, Popov and some authors study the phase transition of the recurrence versus transience for  the model with Bernoulli initial configurations and for the model with drift, see \cite{AMP1,DGHPW,GS,P}.
%Kosygina and Zerner show in \cite{KZ} a zero-one law for the recurrence vs transience for the model on general graphs.
Another interesting feature in the frog model is that it  can be described in the first passage percolation contexts, which is explained below. In fact, Alves, Machado and Popov used this property to prove a shape theorem \cite {AMP}. Moreover, the large deviation estimate for the first passage time is derived in \cite{CKN, K} recently.\\  

The frog model can be defined formally as follows. Let $d\geq 2$ and  $\{({\rm S}_j^x)_{j \in \N}, x \in \Z^d\}$ be independent SRWs such that ${\rm S}_0^x=x$ for any $x \in \Z^d$. For $x,y \in \Z^d$, let 
$$t(x,y) = \inf \{j\in \N_{\geq 0}: {\rm S}_j^x=y\}.$$
The first passage time from $x$ to $y$ is defined by 
$${\rm T}(x,y)= \inf \Big \{\sum_{i=1}^k t(x_{i-1},x_i): x=x_0, \ldots, x_k=y \textrm{ for some } k \Big \}.$$
The quantity ${\rm T}(x,y)$ can be seen as the first time when the frog at $y$ becomes infected assuming that the frog at $x$ was the only infected one at the beginning. For the simplicity of notation, we write ${\rm T}(x)$ instead of ${\rm T}(0,x)$. A path $\gamma=(x_i)^{\ell}_{i=0}$ with $x_0=x$ and $x_{\ell}=y$ is said to be optimal if ${\rm T}(x,y)=\sum_{i=1}^{\ell}t(x_{i-1},x_i)$. For any $x,y\in\Z^d$, such a path certainly exists since ${\rm T}(x,y)$ is a finite natural number almost surely by Lemma \ref{l1}.\\

It has been shown in \cite{AMP} that the first passage time is subadditive, i.e. for any $x,y,z \in \Z^d$
\begin{equation}
  {\rm T}(x,z) \leq {\rm T}(x,y)+{\rm T}(y,z).\label{subadditive}
\end{equation}
The authors of \cite{AMP} also show that the sequence $\{{\rm T}((k-1)z,kz)\}_{k\geq 1}$ is stationary and ergodic for any $z\in \Z^d$. As a consequence of Kingman's subadditive ergodic theorem (see \cite{Ki} or \cite[Theorem 3.1]{AMP}), one has 
\begin{eqnarray} \label{1}
\lim\limits_{n \rightarrow \infty} \frac{{\rm T}(nz)}{n} \rightarrow \kappa_z \quad a.s.,
\end{eqnarray}
 with
 $$\kappa_z = \inf_{n\in \N_{\geq 1}} \frac{\E({\rm T}(nz))}{n}.$$
Furthermore, a shape theorem for the set of active frogs  has been also proved, see \cite[Theorem 1.1]{AMP}.   The convergence \eqref{1}, which can be seen as a law of large numbers, implies that for any $x \in \Z^d$ the first passage time ${\rm T}(x)$ grows linearly in $|x|_1$. A natural question is whether the standard central limit theorem holds for ${\rm T}(x)$. The first task is to understand the behavior of variance of ${\rm T}(x)$.  In \cite{K}, the author proves  some large deviation estimates for ${\rm T}(x)$, see in particular Lemma \ref{l2} below.  As a consequence, one can show that ${\rm Var}({\rm T}(x))=\kO(|x|_1(1+\log |x|_1)^{2A})$, for some constant $A$, see Corollary \ref{cor1}. However, this result is not enough to answer the question on the standard central limit theorem. \\

Our main result is to show that the first passage time has sublinear variance and thus the  central limit theorem with the standard diffusive scaling\footnote{Indeed, it follows from Theorem \ref{mt} and Chebyshev's inequality that $\pp\left(\tfrac{T(x)-\E(T(x))}{\sqrt{\E(T(x))}} \geq t \right)  \leq \frac{C}{t^2 \log |x|_1} \rightarrow 0$ as $|x|_1 \rightarrow \infty$. That rules out the possibility of holding the standard central limit theorem.} is not true. 

\begin{theo} \label{mt}
Let $d\geq 2$. Then there exists a positive constant $C=C(d)$ such that for any $x \in \Z^d$,
$${\rm Var}({\rm T}(x)) \leq \frac{C |x|_1}{\log |x|_1}.$$
\end{theo}
The frog model on $\Z$ (i.e., $d=1$ in our setting) has been carefully investigated by many authors, see e.g., \cite{BR, BRo, GS}. In particular, Commets, Quastel and Ram\'irez \cite{CQR} proved the standard Gaussian fluctuation for the first passage time $\rmt(x)$. As a consequence, ${\rm Var}({\rm T}(x)) \asymp |x|_1$ and the standard central limit theorem for $\rmt(x)$ holds.  We also notice that not only the fluctuation but also the large deviation behavior of $\rmt(x)$ in one dimension is different from that in higher dimensions. Indeed, in the forthcoming paper \cite{CKN}, we and Kubota prove that $\varphi(x)=-\log \pp({\rm T}(x) \geq (1+\varepsilon)\E({\rm T}(x))$ behaves differently when the dimension increases. More precisely, we show that if $d=1$ then $\varphi(x)$ is of order  $\sqrt{|x|_1}$, if $d=2$ then $\varphi(x)$ is of order  $|x|_1/\log |x|_1$ and if $d\geq 3$ then $\varphi(x)$ is of order  $|x|_1$ as $|x|_1\rightarrow \infty$. \\

The sublinearity of variance as in Theorem~\ref{mt}, which is also called the superconcentration,  was first discovered in the  first passage percolation with Bernoulli edge weights by Benjamini, Kalai and Schramm \cite{BKS}. Hence, this result is sometimes called BKS-inequality. Chatterjee  \cite{C14} found the connection among properties of superconcentration, chaos and multiple valleys in, for example, the gaussian polymer model and Sherrington-Kirkpatrick model (see Chapter~5 and 10 in \cite{C14}). This relation is expected to hold in general models. Therefore, the superconcentration is not only an interesting result itself but also an important property to study the structure of optimal paths and the energy landscape.

The method in \cite{BKS} has been improved by Bena\"im and Rossignol in \cite{BRo} to show  the sublinearity of the variance of  ${\rm T}(x)$ in the first passage percolation with a wide class of edges weight distributions, which they called "nearly gamma". Finally, Damron, Hanson and Sosoe in \cite{DHS} generalized the result to all edges  weight distributions with $2+\log$ finite moment.  In this paper, we closely follows the method given in \cite{BKS,BRo,DHS}. However, there are some other difficulties to prove the sublinear variance in the frog models, which will be explained in a sketch of proof below.

\subsection{Sketch of the proof}
First,  we define ${\rm F}_m$ the spatial average of ${\rm T}(x)$  as
$$ {\rm F}_m=\frac{1}{\# {\rm B}(m) } \sum_{z\in {\rm B}(m)} {\rm T}(z,z+x),$$
with  $m=[|x|_1^{1/4}]$ and we prove in Proposition \ref{prop1} that $|\var(\rmt(x))-\var(\rmf_m)|=\kO(|x|_1^{3/4 +\varepsilon})$ for any $\varepsilon >0$. That means we only need to study $\var(\rmf_m)$. As in \cite{BRo, DHS} we consider the martingale decomposition of $\rmf_m$,
\begin{equation*}
{\rm F}_m-\E({\rm F}_m)=\sum_{k=1}^{\infty} \Delta_k,
\end{equation*}
where 
$$\Delta_k=\E({\rm F}_m\mid \kF_k) - \E({\rm F}_m \mid \kF_{k-1}),$$
with $\kF_k$ the sigma-algebra generated by SRWs $\{({\rm S}^{x_i}_j)_{j \in \N}, i =1, \ldots, k\}$ and  $\kF_0$ the trivial sigma-algebra. Note that here we enumerate $\Z^d$ as $\{x_1, x_2, \ldots\}$. As we will see later, with the help of  the weighted logarithmic Sobolev inequality (Lemma \ref{bona}) and the Falik-Smorodnisky inequality (Lemma \ref{l4}), our problem is reduced to prove a series of lemmas \ref{emx}, \ref{t2l} and \ref{mtl}.
%For the proof of these lemmas, we use a control on the maximal weight of paths in locally dependent site-percolation (Lemma \ref{wdsp}).
For illustration, we sketch here the proof of Lemma \ref{emx}, where we show that  as $L\rightarrow \infty$,
\begin{eqnarray} \label{tjol}
\E \left( \max_{\gamma =(y_i)_{i=1}^{\ell} \in \kP_{L}}  \sum \limits_{i=1}^{\ell-1} {\rm T}_1(y_i,y_{i+1})\right) =\kO(L), 
\end{eqnarray}
where $\rmt_1$ is a  modified first passage time, and $\kP_L$ is the set of paths in the box $[-L,L]^d$ with length less than $L$ (see \eqref{def:t1} and section~\ref{sec:notation} for precise definitions). Although the passage times $\{\rmt_1(y_i,y_{i+1})\}_i$ are  concentrated around their means, the correlation among them makes the above problem difficult and interesting. Fortunately, the passage times have the local-dependency property. Indeed, we will show in Lemma \ref{lemttt} that 
\begin{itemize}
	\item [(O1)] for any  $u,v \in \Z^d$,  $M\geq 1$, the event $\{\rmt_1 (u,v) = M \}$ depends only on SRWs $\{(S^x_.) :|x-u|_1 \leq M\}$,
	\item[(O2)] there exist an  integer $C_1\geq 1$ and a constant $\varepsilon_1>0$ such that for any  $u,v \in \Z^d$, $$\pp(\rmt_1(u,v)\geq C_1|u-v|_1) \leq \exp(|u-v|_1^{\varepsilon_1}).$$ 
\end{itemize}
  Starting from these observations, for any path $\gamma=(y_i)_{i=1}^{\ell}$, we consider the following  bound
\begin{equation}
\sum \limits_{i=1}^{\ell-1} {\rm T}_1(y_i,y_{i+1}) \leq  \sum_{M\geq 1} \sum_{k\geq 0} (C_1M+k) a_{M,k}^{\gamma},
\end{equation}
where 
\begin{equation*}
a_{M,0}^{\gamma}=\sum_{y_i \in \gamma} \I(|y_i-y_{i+1}|_1=M, \,\rmt_1(y_i,y_{i+1})\leq C_1M),
\end{equation*}
and for $k\geq 1$,
\begin{equation*}
a_{M,k}^{\gamma}=\sum_{y_i \in \gamma} \I(|y_i-y_{i+1}|_1=M, \,\rmt_1(y_i,y_{i+1})=C_1M+k). 
\end{equation*}
Hence
\begin{equation}
\sum \limits_{i=1}^{\ell-1} {\rm T}_1(y_i,y_{i+1}) \leq C_1|\gamma|_1 +  \sum_{M\geq 1} \sum_{k\geq 1} k a_{M,k}^{\gamma},
\end{equation}
with $|\gamma|_1=\sum_{i=1}^{\ell -1}|y_i-y_{i+1}|$. It is obvious that 
\begin{equation}
a_{M,k}^{\gamma} \leq X_{M,k}(\gamma):= \sum_{y\in \gamma} I_y^{M,k},
\end{equation}
where 
\begin{equation*}
I_y^{M,k}=\I(\exists z : |z-y|_1 \leq M, \rmt_1(y,z)=C_1M+k).
\end{equation*}
Now we arrive at 
\begin{equation} \label{xmk}
\max_{\gamma =(y_i)_{i=1}^{\ell} \in \kP_L}\sum \limits_{i=1}^{\ell-1} {\rm T}_1(y_i,y_{i+1}) \leq C_1L + \max_{\gamma =(y_i)_{i=1}^{\ell} \in \kP_L} \sum_{M\geq 1} \sum_{k\geq 1} kX_{M,k}(\gamma). 
\end{equation}
Here considering the site-percolation on $\Z^d$ generated by  the collection of Bernoulli random variables $\{I_y^{M,k}, y\in \Z^d\}$,   $X_{M,k}(\gamma)$ is the total weight of $\gamma$ on this percolation. Thanks to the observation (O1), the site-percolation  is $(C_1M+k)$-dependent and  by the union bound and (O2),
$$q_{M,k}:=\sup_{y\in \Z^d}\E(I_y^{M,k}) \leq (2(C_1M+k)+1)^d\exp(-(C_1M+k)^{\varepsilon_1}).$$  In the next section,  we prove Lemma \ref{wdsp} to control the maximal weight of paths in locally dependent site-percolation by using a known result for independent site-percolation and tessellation arguments.   In particular, we can show that, with some constant $C>0$,
\begin{align*}
  \E \left(\max_{\gamma \in \kP_L}X_{M,k}(\gamma)\right) &\leq CL (C_1M+k)^d q_{M,k}^{1/d}\\
  &\leq CL (C_1M+k)^{d+1}\exp(-(C_1M+k)^{\varepsilon_1}/d).
\end{align*}
Plugging this estimate into \eqref{xmk}, we get \eqref{tjol}.\\

Our approach seems to be robust and useful for other problems. In particular, using a similar method, we also prove the linearity of the length of optimal paths.

\subsection{The linearity of the lengths of optimal paths}

 Given $x,y\in\Z^d$, let us denote by $\O(x,y)$ the set of all optimal paths from $x$ to $y$.   We simply write $\O(x)$ for $\O(0,x)$. For any path $\gamma=(y_i)_{i=1}^{\ell} \subset \Z^d$, we denote the length of $\gamma$ as $l(\gamma) =\ell$ .  We will prove that the lengths of optimal paths  from $0$ to $x$ grow linearly in $|x|_1$  despite of the fact that optimal paths may have  jumps with size tending to infinity as $|x|_1 \rightarrow \infty$.
\begin{prop} \label{loop} Let $d\geq 2$. Then there exist positive constants $\varepsilon$, $c$ and $C$ such that for any $x \in \Z^d$  
$$	\pp\left( c |x|_1 \leq  \min_{\gamma \in \O(x)} l(\gamma)\leq \max_{\gamma \in \O(x)}  l(\gamma) \leq C|x|_1\right) \geq 1 -e^{-|x|_1^{\varepsilon}}.$$
\end{prop}
%{\bf Sketch of proof}
\subsection{Notation}\label{sec:notation}

\begin{itemize}
\item If $x =(x_1,\ldots, x_d) \in \Z^d$, we denote $|x|_1=|x_1|+\ldots +|x_d|$.
\item For any $n\geq 1$, we denote ${\rm B}(n) = [-n,n]^d$.
\item For any $\ell \geq 1$, we call a sequence of $\ell$ distinct vertices $\gamma=(y_i)_{i=1}^{\ell}$ in $\Z^d$ a path of length $\ell$, we denote $|\gamma|_1=|y_2-y_1|_1+\ldots+|y_{\ell}-y_{\ell -1}|_1$.
\item Given $y=y_i \in \gamma$, we define $\bar{y} =y_{i+1}$ the next point of $y$ in $\gamma$ with the convention that $\bar{y}_{\ell}= y_{\ell}$.
  \item We write $y \sim \bar{y} \in \gamma$ if $\bar{y}$ is the next point of $y$ in $\gamma$.
\item  For $L \geq 1$, we write
$$\kP_L=\{\gamma =(y_i)_{i=1}^{\ell} \subset {\rm B}(L)| ~|\gamma|_1 \leq L,~y_i\neq y_j~if~i\neq j \}.$$
 \item If $f$ and $g$ are two functions, we write $f=\kO(g)$ if there exists a positive constant $C$ such that $f(x) \leq C g(x)$ for any $x$. % Moreover, if the constant $C$ depends on $d$, we write $f=\kO_d(g)$.
 \item We use $C>0$ for a large constant and $\e$ for a small constant. Note that they may change from line to line. 
   \item Given a set $A\subset \Z^d$, we denote by $|A|$ the number of elements of $A$.
\end{itemize}   
\subsection{Organization of this paper}
 The paper is organized as follows. In Section 2, we present some preliminary results including  large deviation estimates on the first passage time and an estimate to control the tail distribution of maximal weight of paths in site-percolation, the introduction and  properties of entropy. In Sections 3 and 4, we prove the main theorem \ref{mt} and Proposition \ref{loop}, respectively.
 \section{Preliminaries}
 \subsection{Large deviation estimates on first passage times} We  present here some useful estimates on the deviation of  first passage times.
\begin{lem} \cite[Proposition 2.4]{K} \label{l1}
	There exists an integer $C_1\geq 1$ and a   positive constant $\varepsilon_1$ such that for any $x, y \in \Z^d$ and  $t \geq C_1|x-y|_1$,
	$$\pp \left({\rm T}(x,y) \geq t  \right) \leq e^{-t^{\varepsilon_1}}.$$
	\end{lem}
We notice that  Lemma \ref{l1} was first proved in \cite[Lemma 4.2]{AMP} for the case $t=C_1|x-y|_1$.  It follows  from Lemma~\ref{l1} that there exists $C>0$ such that for any $x\in\Z^d$,
  \begin{equation} \label{est-expec}
    \begin{split}
      \E {\rm T}(x)\leq C|x|_1.
      \end{split}
  \end{equation}
The following  concentration inequality is derived  in \cite{K}.  
\begin{lem} \cite[Theorem 1.4]{K} \label{l2}
	For any $C >0$, there exist positive constants $a,b$ and $A$ such that for any $x \in \Z^d$ and $(2+\log |x|_1)^{A} \leq t \leq C\sqrt{|x|_1}$,
	$$\pp(|{\rm T}(x)-\E {\rm T}(x)| \geq t\sqrt{|x|_1}) \leq e^{-bt^a}.$$	
\end{lem}
As a direct consequence of Lemmas \ref{l1} and \ref{l2}, we have the following.
\begin{cor} \label{cor1}
	There exists a positive constant $A$ such that 
	$${\rm Var}({\rm T}(x)) =\kO(|x|_1 (1+\log |x|_1)^{2A}).$$
\end{cor}
\begin{proof}
 We take a positive constant $C$ sufficiently large such that Lemma \ref{l1} and  \eqref{est-expec}   hold. By using the fact $\E (X^2)=\int_{0}^\infty 2t \P(X\ge t)dt$ for any non-negative random variable $X$, we get 
  \begin{eqnarray} 
     {\rm Var}({\rm T}(x))&= &\int_{0}^\infty 2t \P(|{\rm T}(x)-\E {\rm T}(x)|\geq t)dt \notag\\
      &= & \int_{0}^{(2+\log |x|_1)^{A} \sqrt{|x|_1}}  2t \P(|{\rm T}(x)-\E {\rm T}(x)|\geq t)dt \notag  \\
      && +\int^{2C|x|_1}_{(2+\log |x|_1)^{A}\sqrt{|x|_1}}  2t \P(|{\rm T}(x)-\E {\rm T}(x)|\geq t)dt +\int^\infty_{2C|x|_1} 2t \P(|{\rm T}(x)-\E {\rm T}(x)|\geq t)dt.  \label{cal-var}
      \end{eqnarray}
  The first term of the right hand side \eqref{cal-var} can be bounded from above by
  $$\int_{0}^{(2+\log |x|_1)^{A} \sqrt{|x|_1}}2t dt\leq (2+\log |x|_1)^{2A}|x|_1.$$
  By Lemma~\ref{l2}, the second term is bounded from above by
  $$|x|_1 \int_0^{\infty} 2t e^{-bt^a}dt=\kO(|x|_1).$$
  Finally, by \eqref{est-expec} and Lemma~\ref{l1}, the third term is bounded from above by 
   \begin{equation*} \label{third-term-est}
          \int^\infty_{2C|x|_1}2t\P({\rm T}(x)\geq t/2)dt\leq  \int^\infty_{2C|x|_1}2t e^{-(t/2)^{\e_1}}dt =\kO(1).
      \end{equation*}
   Combining these estimates, we get the conclusion.
\end{proof}
\begin{lem} \label{l3}
	There exists a positive constant $\e_2$ such that 
	for any $x, y \in \Z^d$ and $M\geq 1$,
	$$\pp({\rm T}(x,y)=t(x,y)=M) \leq e^{-M^{\varepsilon_2}}.$$
\end{lem}	
\begin{proof}
	If $|x-y|_1 \leq M^{2/3}$,  then the result   follows from Lemma \ref{l1}. Assume that $|x-y|_1 \geq M^{2/3}$. Then  a well-known  estimate for the trajectory of random walk (see   \cite[Proposition 2.1.2]{LL}) shows that for some positive constants $c$ and $C$,
	\begin{eqnarray} \label{n2}
	\pp\left(\max_{0\leq j\leq k} |{\rm S}^x_j-x|_1 \geq r\right) \leq C e^{-cr^2/k}.
	\end{eqnarray}
	Therefore,  
	\begin{eqnarray*} 
	\pp(t(x,y)=M) \leq \pp \left(\max_{0\leq j\leq M} |{\rm S}^x_j-x|_1 \geq M^{2/3}\right) \leq C e^{-cM^{1/3}}.
	\end{eqnarray*}
\end{proof}
\subsection{The maximal weight of paths in  site-percolation} 
\subsubsection{The case of independent percolation}
Let  $\{I_x\}_{x\in\Z^d}$ be a collection of  independent  random variables such that $\pp(I_x=1)=1-\pp(I_x=0)=p_x \leq p$ with a parameter $p\in[0,1]$ for all $x\in \Z^d$. For any  $A \subset \Z^d $, we define  the weight of $A$ as
$X(A)=\sum_{x\in A} I_x$.  The  maximal weight of paths in $\kP_L$ is defined as
$$X_L= \max_{\gamma \in \kP_L} X(\gamma).$$
The tail distribution  and  expectation of $X_L$ can be controlled  by the following lemma.
\begin{lem}\label{wsp}
  % \cite[Lemma 6.8]{DHS} 
	There exist  positive constants $A_1$ and $A_2$ such that the following statements hold.
	\begin{itemize}
		\item [(i)] If $\min \{sLp^{1/d}, s\} \geq A_1$ then
		$$\pp\left(X_L \geq s L p^{1/d}\right) \leq \exp\left(-sLp^{1/d}/2\right).$$
		\item [(ii)] For any $p\in(0,1)$ and $L\geq 1$,
	          $$\E\left(X_L\right) \leq A_2Lp^{1/d}.$$
	\end{itemize}
\end{lem}  
\begin{proof} We start by recalling a result   in \cite{DHS} on the maximal weight of lattice animals (i.e., connected sets containing $0$). Define 
\begin{equation}
N_L=\sup \{X(A): 0\in A, \textrm{$A$ is connected, } |A|\leq L+1  \}.
\end{equation}
In Lemma 6.8 in \cite{DHS}, the authors show that there exist  positive constants $A_1'$ and $A_2'$ such that 
\begin{itemize}
	\item [(a)] if $Lp^{1/d} >1$ and $s\geq A_1'$, then
	$$\pp\left(N_L \geq s L p^{1/d}\right) \leq \exp\left(-sLp^{1/d}/2\right),$$
	\item [(b)] for any $p\in(0,1)$ and $L\geq 1$,
	$$\E\left(N_L\right) \leq A_2'Lp^{1/d}.$$
\end{itemize}
Let $\gamma =(y_i)_{i=1}^{\ell} \in \kP_L$. Then $\gamma \subset {\rm B(L)}$ and $\sum_{i=2}^{\ell} |y_{i}-y_{i-1}|_1 \leq L$.  Thus $\sum_{i=1}^{\ell} |y_{i}-y_{i-1}|_1 \leq (d+1)L$ with $y_0=0$. Considering  shortest paths from $y_{i-1}$ to $y_i$ for $1 \leq i\leq \ell$ in the lattice $\Z^d$, there exists a connected set $A\subset\Z^d$ such that $\gamma\subset A$ and $|A| \leq 1+ \sum_{i=1}^{\ell} |y_{i}-y_{i-1}|_1 \leq (d+1)L+1$. Therefore $X(\gamma) \leq X(A) \leq N_{(d+1)L}$ for all $\gamma \in \kP_L$. Hence 
\begin{equation} \label{xleqn}
X_L \leq N_{(d+1)L}.
\end{equation} 
Using \eqref{xleqn} and (b), we obtain (ii). We now prove that (i) holds for $A_1=(d+1)A_1'$ with $A_1'$ as in (a). Let us  denote by $\pp_p$ the probability measure of site-percolation with density $p$.  Using \eqref{xleqn},
\begin{eqnarray}
\pp_{p} \left(X_L > s Lp^{1/d}\right) \leq \pp_p \left(N_{(d+1)L} > s Lp^{1/d}\right) = \pp_p \left(N_{(d+1)L} > \tfrac{s}{d+1} (d+1)Lp^{1/d}\right).
\end{eqnarray}
Suppose $\min \{sLp^{1/d}, s\} \geq A_1$. If $(d+1)Lp^{1/d} >1$, then using (a) and $\tfrac{s}{d+1} \geq A_1'$,
\begin{equation} \label{ac1}
\pp_p \left(N_{(d+1)L} > \tfrac{s}{d+1} (d+1)Lp^{1/d}\right) \leq \exp\left(-sLp^{1/d}/2\right).
\end{equation}
For the case $(d+1)Lp^{1/d} \leq 1$, we define $q=L^{-d}$. Then $p<q$ and $(d+1)Lq^{1/d}>1$. Thus using the monotonicity of $\pp_p$ in $p$ and (a),
\begin{eqnarray} \label{ac2}
\pp_p \left(N_{(d+1)L} > sLp^{1/d}\right) &\leq& \pp_q \left(N_{(d+1)L} > sLp^{1/d}\right)  \notag \\
&=& \pp_q \left(N_{(d+1)L} > \tfrac{sLp^{1/d}}{d+1} (d+1)Lq^{1/d}\right) \leq \exp\left(-sLp^{1/d}/2\right),
\end{eqnarray}
since $\tfrac{sLp^{1/d}}{d+1} \geq A_1'$. Combining \eqref{ac1} and \eqref{ac2} we get  (i).
\end{proof}

\subsubsection{The case of $M$-dependent percolation}

Given $M \geq 1$, let $\{I_x, x \in \Z^d\}$ be a collection of Bernoulli random variables  such that
\begin{itemize}
	\item [(E1)] $\{I_x, x \in \Z^d\}$ is $M$-dependent, i.e., for all $x\in \Z^d$, the variable $I_x$ is independent of all variables $\{I_y: |y-x|_1 >M\}$, 
	\item[(E2)]  $q_M = \sup_{x \in \Z^d} \E(I_x)  \leq (3M+1)^{-d}$.
\end{itemize}
For any path $\gamma$ in $\Z^d$, we also define 
\begin{equation}
X(\gamma) = \sum_{x \in \gamma}I_x, \qquad X_L= \max_{\gamma \in \kP_L} X(\gamma).
\end{equation}
\begin{lem} \label{wdsp}
	Let $M\geq 1 $ and $\{I_x, x \in \Z^d\}$ be a collection of random variables  satisfying (E1) and (E2). Then there exists a  positive constant $C=C(d)$  such that 
	\begin{itemize}
		\item [(i)] for any $L \geq 1$,
		\begin{equation}
		\E \left(X_L\right) \leq  C L M^{d+1} q_M^{1/d},
		\end{equation} 
		\item[(ii)] if $n \geq  C M^d \max \{1,  MLq_M^{1/d}\}$, then
		\begin{equation}
		\pp \left(X_L \geq n \right) \leq  2^d \exp (-n/(16M)^d).  
		\end{equation}
	\end{itemize}
\end{lem}
\begin{proof} 
	For each $M\geq 1$, let us consider a standard tessellation of $\Z^d$ constructed as follows. Enumerate $\{0,1\}^d$ as $\{w_i, i=1,\ldots, 2^d\}$. Then for any $i\in\{1,\ldots,2^d\}$ and $z\in\Z^d$, we define
	\begin{equation}
	B_{i,z}^M= 3M(w_i+2z)+[0,3M]^d.
	\end{equation} 
Then $(B_{i,z}^M)_{i,z}$ are boxes of side length $3M$ satisfying
	\begin{itemize}
		\item[(a)] for all $y\in \Z^d$, there exists ${\rm B}^M_{i,z}$ containing $y$,
	%	\item[(b)] if $(i,z) \neq (i',z')$ then ${\rm B}^M_{i,z} \cap {\rm B}^M_{i',z'} = \varnothing$,
		\item[(b)] for any $i=1,\ldots,2^d$, the boxes in the $i$-th group,  $({\rm B}^M_{i,\cdot})$ satisfy that  the distance $|\cdot|_1$ between two arbitrary boxes is larger than $3M$.
	\end{itemize}
	 For $i=1,\ldots, 2^d$,  $z \in \Z^d$ and $\gamma \in \kP_L$,  define
	$$ X_{i,z}(\gamma) = X(\gamma \cap \rmb) = \sum_{x \in \gamma \cap  \rmb} I_x.$$
	Then by (a), 
	\begin{eqnarray} \label{25}
	X(\gamma) \leq   \sum_{i=1}^{2^d} \sum_{z\in \Z^d }  X_{i,z}(\gamma).
	\end{eqnarray}
	It is clear that for all $i=1, \ldots, 2^d$,
	\begin{eqnarray} \label{ayiz}
 \sum_{z\in \Z^d } X_{i,z}(\gamma)  \leq (3M+1)^d \sum_{z\in \eta^{i,M}} Y^i_z,
	\end{eqnarray}
	where	$\eta^{i,M}$ is the projected path of $\gamma$ defined by 
	$$\eta^{i,M}=\{z\in \Z^d:  \gamma \cap \rmb \neq \varnothing \},$$
	and 
		$$Y^i_z=\I \left(\textrm{$\exists x   \in {\rm B}^M_{i,z}$ such that $I_x=1$}\right).$$
	Since  $ \gamma \in \kP_L$, we have
	$$\eta^{i,M} \in \kP_{\lceil L/(3M)\rceil}.  $$
	Hence,
	\begin{eqnarray}
	\sum_{z\in \eta^{i,M} } Y^i_z \leq\max_{\eta \in \kP_{\lceil L/(3M) \rceil}} \sum_{z\in \eta } Y^i_z =:  X_{L,M}^i.
	\end{eqnarray}
	Combining this inequality with \eqref{25} and \eqref{ayiz} yields that
	\begin{eqnarray} \label{29}
X_L = \max_{\gamma \in \kP_L}	\sum_{i=1}^{2^d} \sum_{z\in\Z^d}X_{i,z}(\gamma) \leq (3M+1)^d \sum_{i=1}^{2^d} X_{L,M}^i.
	\end{eqnarray}
	By (b) and (E1),  $(Y^i_z)_{z\in \Z^d}$ are independent  Bernoulli random variables. Moreover,  by the union bound and (E2)  
	\begin{equation} \label{pm}
	\begin{split}
	p_M: = \sup_{(i,z)}\E(Y^i_z) &= \sup_{(i,z)}\pp(\exists x \in {\rm B}^M_{i,z}: I_x =1)\\
	&\leq (3M+1)^{d} q_M\leq 1. 
	\end{split}
	\end{equation}
	Now applying Lemma \ref{wsp} to the set of random variables $(Y^i_z)_{z\in \Z^d}$ and the set of paths $\kP_{\lceil L/(3M) \rceil}$, we get 
	\begin{eqnarray} \label{t1}
	\E(X_{L,M}^i) &\leq& A_2\lceil L/(3M) \rceil p_M^{1/d},
	\end{eqnarray}
	with $A_2$ as in Lemma \ref{wsp} (ii). 
	Combining \eqref{29}, \eqref{pm} and \eqref{t1} gives 
	\begin{eqnarray} \label{42}
	\E(X_L) &\leq& A_22^d \lceil L/(3M) \rceil (3M+1)^{d} p_M^{1/d}  \notag \\
	& \leq &C L M^{d+1} q_M^{1/d}, 
	\end{eqnarray}
for some $C=C(d)$. This proves (ii). We now turn to prove (i). Observe that by \eqref{29}, for all $n$
\begin{eqnarray}
\pp\left(X_L \geq n\right) \leq \pp\left(\sum_{i=1}^{2^d} X_{L,M}^i  \geq n/ (4M)^d\right) \leq  \sum_{i=1}^{2^d} \pp\left( X_{L,M}^i  \geq n/ (8M)^d\right). 
\end{eqnarray}
By Lemma \ref{wsp} (i), for all $i=1,\ldots, 2^i$,
\begin{equation*}
\pp\left( X_{L,M}^i  \geq n/ (8M)^d\right) = \pp\left( X_{L,M}^i  \geq \frac{n}{(8M)^d \lceil L/(3M) \rceil p_M^{1/d}}  \lceil L/(3M) \rceil p_M^{1/d} \right)   \leq \exp\left(-n/(2(8M)^{d})\right),
\end{equation*}
provided that 
\begin{equation} \label{na1}
\min \Big \{\frac{n}{(8M)^d}, \frac{n}{(4M)^d \lceil L/(3M) \rceil p_M^{1/d}} \Big \} \geq A_1,
\end{equation}
with $A_1$ as in Lemma \ref{wsp} (i).  Using \eqref{pm},  the condition \eqref{na1} follows if 
\begin{equation}
n \geq A_1 (8M)^d \max \{1, \lceil L/(3M) \rceil (3M+1)q_M^{1/d}\},
\end{equation}
which is satisfied if 
\begin{equation} \label{na1m}
n \geq  C M^d \max \{1,  MLq_M^{1/d}\},
\end{equation}
for some $C=C(A_1,d)$. In conclusion, if \eqref{na1m} holds then
\begin{eqnarray}
\pp\left(X_L \geq n\right) \leq 2^d \exp (-n/(16M)^d). 
\end{eqnarray}
\end{proof}

\subsection{Entropy}  Let $(\Omega, \kF, \mu)$ be a probability space and $X\in L^1(\Omega, \mu)$ an non-negative random variable. Then the entropy of $X$ with respect to $\mu$ is defined as 
$${\rm Ent}_{\mu}(X)=\E_{\mu} (X \log X) - \E_{\mu}(X) \log \E_{\mu}(X).$$
Note that by Jensen's inequality, ${\rm Ent}_{\mu}(X) \geq 0$. The following tensorization  property of entropy is well-known and we refer the reader to \cite{BLM} for the proof. 
\begin{lem} \cite[Theorem 4.22]{BLM} \label{ent}
	Let $X$ be a non-negative $L^2$ random variable on a product space 
	$$\left(\prod_{i=1}^{\infty} \Omega_i, \kF, \mu= \prod_{i=1}^{\infty} \mu_i\right),$$
	where $\kF= \bigvee_{i=1}^{\infty} \kG_i$, and each triple $(\Omega_i, \kG_i, \mu_i)$ is a probability space. Then 
	$${\rm Ent}_{\mu}(X) \leq \sum_{i=1}^{\infty} \E_{\mu} {\rm Ent}_i(X),$$
	where ${\rm Ent}_i(X)$ is the entropy of $X(\omega)=X((\w_1, \ldots, \w_i, \ldots))$ with respect to $\mu_i$, as a function of the $i$-th coordinate (with all other values fixed).
\end{lem} 
The following weighted logarithmic Sobolev inequality will be useful for estimating the entropy of martingale difference. 

\begin{lem} \cite[Lemma 2.6]{MS} \label{bona}
	Assume that $k \geq 2$. Let $f:\{1,\ldots,k\} \mapsto \R$ be a function and $\nu$ be the uniform distribution on $\{1,\ldots,k\}$. Then 
	$${\rm Ent}_{\nu}(f^2) \leq k {\rm E}((f(U)-f(\tilde{U}))^2),$$
	where ${\rm E}$ is the expectation with respect to two independent random variables $U, \tilde{U}$, which have the same distribution $\nu$. 
\end{lem}

\section{Proof of Theorem \ref{mt}}

 \subsection{ Spatial average of the first passage time} We consider a spatial average of ${\rm T}(x)$ defined as 
 \begin{equation} \label{def:fm}
 {\rm F}_m=\frac{1}{\# {\rm B}(m) } \sum_{z\in {\rm B}(m)} {\rm T}(z,z+x),
 \end{equation}
 where 
 $$m=[|x|_1^{1/4}].$$ 
 \begin{prop} \label{prop1}
 For any $\varepsilon >0$, it holds that 
 	$$|{\rm Var}({\rm T}(x))-{\rm Var}({\rm F}_m)| =\kO(|x|_1^{3/4 +\varepsilon} ).$$
 \end{prop}
\begin{proof}
	For any variables $X$ and $Y$, by writing $\hat{X}=X-\E(X)$ and $||X||_2=(\E(X^2))^{1/2}$ and using the Cauchy-Schwarz inequality, we get
	\begin{eqnarray} \label{e1}
	|{\rm Var}(X)-{\rm Var}(Y)| = |E(\hat{X}^2-\hat{Y}^2)| &\leq &||\hat{X}+\hat{Y}||_2 ||\hat{X}-\hat{Y}||_2 \notag 	\\
	&\leq& (||\hat{X}||_2+||\hat{Y}||_2)||\hat{X}-\hat{Y}||_2.
	\end{eqnarray}
We aim to apply  \eqref{e1} for ${\rm T}(x)$ and ${\rm F}_m$. Observe that  
\begin{eqnarray} \label{e2}
||\hat{\rm F}_m||_2 \leq \frac{1}{\#{\rm B}(m)} \sum_{z\in B_m} ||\hat{\rm T}(z,z+x)||_2= ||\hat{\rm T}(0,x)||_2,
\end{eqnarray}	
by translation invariance. By Corollary \ref{cor1},
\begin{eqnarray} \label{e3}
||\hat{\rm T}(0,x)||_2=\sqrt{{\rm Var}({\rm T}(x))} =\kO(|x|_1^{1/2}(1+\log |x|_1)^{A}).
\end{eqnarray}
Using the subadditivity \eqref{subadditive},
\begin{eqnarray}
 ||\hat{\rm T}(0,x)-\hat{\rm F}_m||_2^2 &=& ||{\rm T}(x)-{\rm F}_m||_2^2 \notag \\
 &=& \frac{1}{\#{\rm B}(m)^2} \left\|\sum_{z\in {\rm B}(m)} ({\rm T}(x)-{\rm T}(z,z+x))\right\|_2^2 \notag\\
&\leq&  \frac{1}{\#{\rm B}(m)^2} \left\|\sum_{z\in {\rm B}(m)} ({\rm T}(z)+{\rm T}(x,z+x)+{\rm T}(z,0)+{\rm T}(z+x,x))\right\|_2^2 .\notag \\
&\leq& \frac{4}{\#{\rm B}(m)} \sum_{z\in {\rm B}(m)} \Big [\E{\rm T}(z)^2+\E{\rm T}(x,z+x)^2+\E{\rm T}(z,0)^2+\E{\rm T}(z+x,x)^2 \Big ]  \notag \\
&\leq& 16 \max_{z\in {\rm B}(m)} \E{\rm T}(z)^2,  \label{tox-fm}
\end{eqnarray}
where we used the following inequality in the $4$-th line, $$ \Big(\sum_{j\in \Lambda} a_j+b_j+c_j+d_j \Big)^2 \leq 4 |\Lambda|\sum_{j\in \Lambda} (a_j^2+b_j^2+c_j^2+d_j^2), $$ 
and we used the translation invariant in the last line.
\begin{comment}
Using the Cauchy-Shwartz inequality and the translation invariance, this is further bounded from above by
\begin{equation}
  \begin{split}
& \frac{1}{\#{\rm B}(m)} \left(\left\|\sum_{z\in {\rm B}(m)} {\rm T}(z)\right\|_2 +\left\|\sum_{z\in {\rm B}(m)} {\rm T}(x,x+z) \right\|_2\right)  \notag \\
    = &\frac{2}{\#{\rm B}(m)} \left\|\sum_{z\in {\rm B}(m)} {\rm T}(z)\right\|_2 \\
    \leq & 2 \left\|\max_{z\in {\rm B}(m)} {\rm T}(z)\right\|_2. \label{e4}
\end{split}
\end{equation}
By   Lemma \ref{l1} and the union bound, for all $k\geq C_1|x|_1^{1/4}$, with $C_1$ as in Lemma \ref{l1}
$$\pp \left(\max_{z\in {\rm B}(m)}{\rm T}(z) \geq k  \right) \leq (\#{\rm B}(m))  e^{-k^{\varepsilon_1}}.$$
	Therefore, by a similar argument as in Corollary~\ref{cor1}, we have
	\begin{eqnarray} \label{e5}
	\E \left(\max_{z\in {\rm B}(m)}{\rm T}(z)^2\right)& \leq& C_1^2 |x|_1^{1/2} + (\#{\rm B}(m))\sum_{k\geq C_1|x|_1^{1/4}} k^2 e^{-k^{-\varepsilon_1} }  \notag \\
	&=& \kO(|x|_1^{1/2}).
	\end{eqnarray}
	Combining \eqref{e1}--\eqref{e5}, we get the desired result.
\begin{lem}
  Let $\{X_i\}^n_{i=1}$ be a sequence of random variables. Then,
  \begin{equation}
    \E\left[\left(\sum^n_{i=1}X_i\right)^2\right]\leq n \left(\sum^n_{i=1}\E X_i^2\right).
    \end{equation}
\end{lem}
\begin{proof}
  By the Cauthy-Schwarz inequality, we have
  \begin{align*}
    \left(\sum_{i=1}^n X_i\right)^2&\leq \left(\sum_{i=1}^n 1\right)\left(\sum_{i=1}^n X_i^2\right)\\
    &=n \sum_{i=1}^n X_i^2.
  \end{align*}
  By the linearity of the expectation, the proof is completed.
\end{proof}
\end{comment}
Since $\E{\rm T}(z)^2=\var( {\rm T}(z))+(\E \rmt(z))^2$, by using \eqref{tox-fm}, \eqref{est-expec} and Corollary~\ref{cor1},
\begin{eqnarray} \label{e5}
   ||\hat{\rm T}(0,x)-\hat{\rm F}_m||_2^2 = \kO(m^2)= \kO(|x|_1^{1/2}).
	\end{eqnarray}
	Combining \eqref{e1}--\eqref{e5}, we get the desired result.
\end{proof}
 \subsection{Martingale decomposition of ${\rm F}_m$ and the proof of Theorem \ref{mt}}
Enumerate the vertices of $\Z^d$ as $x_1,x_2, \ldots$. We consider the martingale decomposition of ${\rm F}_m$ as follows
\begin{equation}\label{Martingale-decom}
  {\rm F}_m-\E({\rm F}_m)=\sum_{k=1}^{\infty} \Delta_k,
  \end{equation}
where 
$$\Delta_k=\E({\rm F}_m\mid \kF_k) - \E({\rm F}_m \mid \kF_{k-1}),$$
with $\kF_k$ the sigma-algebra generated by SRWs $\{({\rm S}^{x_i}_j)_{j \in \N}, i =1, \ldots, k\}$ and  $\kF_0$ the trivial sigma-algebra. In \cite{DHS}, using the Falik-Samorodnitsky lemma, the authors give  an upper bound for the variance of ${\rm F}_m$ in term of ${\rm Ent}(\Delta_k^2)$, and $\E(|\Delta_k|)$. 
\begin{lem} \cite[Lemma 3.3]{DHS} \label{l4}
	We have 
	$$\sum_{k\geq 1} {\rm Ent}(\Delta_k^2) \geq {\rm Var} ({\rm F}_m) \log \left[ \frac{{\rm Var}({\rm F}_m)}{\sum_{k\geq 1} (\E(|\Delta_k|))^2} \right].$$
\end{lem}
Now, our main task is to estimate ${\rm Ent}(\Delta_k^2)$ and $\E(|\Delta_k|)$.
\begin{prop} \label{prop2}
	As $|x|_1$ tends to infinity,
	\begin{itemize}
		\item [(i)] $$\sum_{k\geq 1} {\rm Ent}(\Delta_k^2) = \kO(|x|_1). $$
		\item [(ii)]  $$\sum_{k\geq 1} (\E(|\Delta_k|))^2 = \kO \left(|x|_1^{\tfrac{5-d}{4}}\right).$$
	\end{itemize}
	
\end{prop}
{\it Proof of Theorem \ref{mt} assuming Proposition \ref{prop2}}.
 Since  $d\geq 2$, Proposition \ref{prop2} (ii) implies that $\sum_{k\geq 1} (\E(|\Delta_k|))^2 = \kO \left(|x|_1^{3/4}\right)$. Therefore, using  Propositions \ref{prop1}, \ref{prop2} and Lemma \ref{l4}, for any $\varepsilon >0$, there exists a positive constant $C$ such that
\begin{eqnarray} \label{e6}
  {\rm Var}({\rm T}(x)) & \leq&  {\rm Var}({\rm F}_m)+ C |x|_1^{3/4 + \varepsilon} \notag \\
  &\leq& C\left(|x|_1^{3/4 + \varepsilon} +  |x|_1 \left[\log \left[\frac{{\rm Var}({\rm F}_m)}{|x|_1^{3/4}}\right] \right]^{-1} \right).  
\end{eqnarray}
If ${\rm Var}({\rm F}_m) \leq |x|_1^{7/8}$, then ${\rm Var}({\rm T}(x)) = \kO(|x|_1^{7/8})$ and Theorem \ref{mt} follows. Otherwise, if ${\rm Var}({\rm F}_m) \geq |x|_1^{7/8}$, using  \eqref{e6}  we get that ${\rm Var}({\rm T}(x))=\kO(|x|_1/\log|x|_1)$ and Theorem \ref{mt} also follows. \hfill $\square$

\subsection{Proof of Proposition \ref{prop2}} 
By the definition of $\Delta_k$, we have
\begin{eqnarray} \label{vk}
|\Delta_k| &=& \frac{1}{\# {\rm B}(m)}  \left| \E\left[ \sum_{z\in {\rm B}(m)} {\rm T}(z,z+x) \mid \kF_k\right] -  \E\left[ \sum_{z\in {\rm B}(m)} {\rm T}(z,z+x) \mid \kF_{k-1}\right] \right| \notag \\
&\leq & \frac{1}{\# {\rm B}(m)} \sum_{z\in {\rm B}(m)} \Big| \E\left[  {\rm T}(z,z+x) \mid \kF_k\right] -  \E\left[  {\rm T}(z,z+x) \mid \kF_{k-1}\right] \Big|.
\end{eqnarray}
We precise the dependence of first passage times on trajectories of SRWs by writing
$${\rm T}(u,v)={\rm T}(u,v,({\rm S}^{x_i}_{.})_{i \in \N}).$$
For any $k$, let us define 
$$X_k(u,v)=\E({\rm T}(u,v)\mid \kF_k).$$
Then $X_k(u,v)$ is a function of  trajectories of $({\rm S}^{x_i}_.)_{i\leq k}$, so we write 
$$X_k(u,v)=X_k(u,v)[({\rm S}^{x_i}_.)_{i < k}, ({\rm S}^{x_k}_.)].$$
Let $(\tilde{S}^{x}_.)_{x\in\Z^d}$ be an independent copy of $({\rm S}^{x}_.)_{x\in\Z^d}$. We observe that
\begin{eqnarray} \label{ae}
\E(|X_k(u,v)-\E^k(X_k(u,v))|)\leq \E^{<k} \E^{k} \tilde{\E}^k (|X_k(u,v)-\tilde{X}_k(u,v)|),
\end{eqnarray}
where 
$$\tilde{X}_k(u,v)=X_k(u,v)[({\rm S}^{x_i}_.)_{i < k}, (\tilde{S}^{x_k}_.)],$$
 and $\E^{<k}, \E^{k}$, and  $\tilde{\E}^k$ denote the expectations with respect to SRWs $({\rm S}^{x_i}_.)_{i<k}$, $({\rm S}^{x_k}_.)$ and $(\tilde{S}^{x_k}_.)$ respectively. Then the inequality \eqref{ae} becomes 
\begin{equation} \label{9}
\E\Big| \E\left[  {\rm T}(z,z+x) \mid \kF_k\right] -  \E\left[  {\rm T}(z,z+x) \mid \kF_{k-1}\right] \Big| \leq  \E \tilde{\E}^k\Big| {\rm T}(z,z+x) - \tilde{\rm T}_{x_k}(z,z+x) \Big|,
\end{equation}
where for $u,v \in \Z^d$ and $k\geq 1$ 
$$\tilde{\rm T}_{x_k}(u,v)= {\rm T}(u,v)[({\rm S}^{x_i}_.)_{i<k}, (\tilde{S}^{x_k}_.),({\rm S}^{x_i}_.)_{i>k}]. $$
By the symmetry ${\rm T}(z,z+x) - \tilde{\rm T}_{x_k}(z,z+x)$,
\begin{eqnarray} \label{10}
&& \E  \tilde{\E}^k\Big| {\rm T}(z,z+x) - \tilde{\rm T}_{x_k}(z,z+x) \Big|  \notag \\
&=& 2 \E \tilde{\E}^k \left( (\tilde{\rm T}_{x_k}(z,z+x)-{\rm T}(z,z+x)) \mathbb{I}(\tilde{\rm T}_{x_k}(z,z+x) \geq {\rm T}(z,z+x)) \right).
\end{eqnarray}
For any $u,v\in \Z^d$, we choose an  optimal path for ${\rm T}(u,v)$ with a deterministic rule breaking ties and denote it by $\gamma_{u,v}$. We observe that if $x_k \not \in \gamma_{u,v}$ then $\tilde{\rm T}_{x_k}(u,v) \leq {\rm T}(u,v)$. Otherwise, if $x_k  \in \gamma_{u,v}$, then 
\begin{equation} \label{ae1}
{\rm T}(u,v)={\rm T}(u,x_k)+{\rm T}(x_k,\bar{x}_k) + {\rm T}(\bar{x}_k,v),
\end{equation}
where $\bar{x}_k$ is the next point of $x_k$ in $\gamma_{u,v}$ (recall also that we denote by $y \sim \bar{y} \in \gamma$ if $\bar{y}$ is the next point of $y$ in $\gamma$).  Due to the subadditivity,
\begin{eqnarray} \label{ae2}
  \tilde{\rm T}_{x_k}(u,v) &\leq& \tilde{\rm T}_{x_k}(u,x_k) + \tilde{\rm T}_{x_k}(x_k,\bar{x}_k)+\tilde{\rm T}_{x_k}(\bar{x}_k,v).
  \end{eqnarray}
It is clear that any optimal path for ${\rm T}(u,x_k)$ does not use the simple random walk $({\rm S}^{x_k}_{\cdot})$. Hence,
\begin{equation} \label{ae3}
\tilde{\rm T}_{x_k}(u,x_k) \leq {\rm T}(u,x_k).
\end{equation}
In addition, since $\bar{x}_k$ is the next point of $x_k$ in  $\gamma_{u,v}$, the optimal path for ${\rm T}(\bar{x}_k,v)$ does not use the simple random walk $({\rm S}^{x_k}_{\cdot})$. Thus
\begin{equation} \label{ae4}
\tilde{\rm T}_{x_k}(\bar{x}_k,v) \leq {\rm T}(\bar{x}_k,v).
\end{equation}
It follows from \eqref{ae1}--\eqref{ae4} that
$$\tilde{\rm T}_{x_k}(u,v) -{\rm T}(u,v) \leq \tilde{\rm T}_{x_k}(x_k,\bar{x}_k). $$
Therefore, we have
\begin{eqnarray} \label{11}
&& (\tilde{\rm T}_{x_k}(z,z+x)-{\rm T}(z,z+x)) \mathbb{I}(\tilde{\rm T}_{x_k}(z,z+x) \geq {\rm T}(z,z+x)) \notag \\
& \leq & \tilde{\rm T}_{x_k}(x_k,\bar{x}_k) \mathbb{I}(x_k \in \gamma_{z,z+x}).
\end{eqnarray}
 We notice here that the complete notation of  $\bar{x}_k $ should be $ \bar{x}_k(\gamma_{z,z+x})$ to highlight the dependence of $\bar{x}_k$ on the path $\gamma_{z,z+x}$. However, for the simplicity of notation, we shortly write it by $\bar{x}_k$ when the fact $x_k \in \gamma_{z,z+x}$ is precise.  Combining \eqref{vk}, \eqref{9}, \eqref{10} and \eqref{11}, we get 
\begin{eqnarray} \label{s15}
\E(|\Delta_k|) &\leq& \frac{2}{\#{\rm B}(m)}   \E^{\otimes 2} \left(\sum \limits_{z\in {\rm B}(m)}  \tilde{\rm T}_{x_k}(x_k,\bar{x}_k) \mathbb{I}(x_k \in \gamma_{z,z+x})\right) \notag\\
&=& \frac{2}{\#{\rm B}(m)}   \E^{\otimes 2} \left(\sum \limits_{z\in {\rm B}(m)} \tilde{\rm T}_{x_k-z}(x_k-z, \overline{x_k-z} ) \I(x_k-z \in \gamma_{0,x})\right) \notag \\
&=& \frac{2}{\#{\rm B}(m)}   \E^{\otimes 2} \left(\sum \limits_{y\in x_k-{\rm B}(m)} \tilde{\rm T}_{y}(y, \bar{y} ) \I(y \in \gamma_{0,x})\right) \notag \\
&=& \frac{2}{\#{\rm B}(m)} \sum_{L\geq 0}  \E^{\otimes 2} \left(\sum \limits_{y\in x_k-{\rm B}(m)} \tilde{\rm T}_{y}(y, \bar{y} ) \I(y \in \gamma_{0,x}) \I(\kE_{k,L})\right),
\end{eqnarray}
where $\E^{\otimes 2}$ is the expectation with respect to two independent collections of SRWs $({\rm S}^{x_i}_.)_{i\in \N}$ and $(\tilde{S}^{x_i}_.)_{i\in \N}$, and we define
$$\kE_{k,L}=\left\{\sum_{y\in\gamma_{0,x} \cap (x_k-{\rm B}(m)) }|y-\bar{y}|_1=L\right\}. $$ 
Notice that for the second equation, we have used the invariant  translation. Let us define
$${\rm T}^{[z]}(u,v) = \inf \Big \{\sum_{l=1}^k t(y_{l-1},y_l): u=y_0, \ldots, y_k=v, y_l \neq z \, \forall \, l \geq 1,  \textrm{ for some } k \Big \},$$ 
as the first passage time from $u$ to $v$ not using the frog at $z$, and set
\begin{equation} \label{def:t1}
{\rm T}_1 (u,v)=\max_{z:~|z-u|_1=1}  {\rm T}^{[u]}(z,v)+1. 
\end{equation}
Then, it holds that  
\begin{equation} \label{ae5}
\tilde{\rm T}_u(u,v)\leq {\rm T}_1 (u,v).
\end{equation}
Using \eqref{ae5}, we obtain 
\begin{eqnarray*}
  \sum \limits_{y\in (x_k-{\rm B}(m))} \tilde{\rm T}_{y}(y, \bar{y} ) \I(y \in \gamma_{0,x}) \I(\kE_{k,L}) &\leq& \max_{\substack{\gamma =(y_i)_{i=1}^{\ell}\subset (x_k-{\rm B}(m+L))\\|\gamma|_1 \leq L}}  \sum \limits_{i=1}^{\ell-1}  \tilde{\rm T}_{y_i}(y_i,y_{i+1}) \I(\kE_{k,L})\\
&\leq&   \max_{\substack{\gamma =(y_i)_{i=1}^{\ell}\subset (x_k-{\rm B}(m+L))\\|\gamma|_1 \leq L}}  \sum \limits_{i=1}^{\ell-1}{\rm T}_1(y_i,y_{i+1}) \I(\kE_{k,L}).
\end{eqnarray*}
Therefore, with $C_1 \geq 1$ as in Lemma \ref{l1},
\begin{eqnarray} \label{s16}
&&\sum_{L= 0}^{4dC_1m}  \E^{\otimes 2} \left(\sum \limits_{y\in (x_k-{\rm B}(m+L))} \tilde{\rm T}_{y}(y, \bar{y} ) \I(y \in \gamma_{0,x}) \I(\kE_{k,L})\right) \notag \\
 & \leq&
  \E \left(\max_{\substack{\gamma =(y_i)_{i=1}^{\ell}\subset (x_k-{\rm B}(8dC_1 m))\\|\gamma|_1 \leq 4dC_1m}}  \sum \limits_{i=1}^{\ell-1} {\rm T}_1(y_i,y_{i+1}) \right) \notag\\
 &=&\E \left(\max_{\substack{\gamma =(y_i)_{i=1}^{\ell}\subset {\rm B}(8dC_1 m)\\|\gamma|_1 \leq 4dC_1m}}  \sum \limits_{i=1}^{\ell-1}  {\rm T}_1(y_i,y_{i+1}) \right) \notag \\
  &\leq & \E \left(\max_{\gamma =(y_i)_{i=1}^{\ell} \in \kP_{8dC_1m}}  \sum \limits_{i=1}^{\ell-1}  {\rm T}_1(y_i,y_{i+1}) \right),
\end{eqnarray}
and 
\begin{eqnarray} \label{s17}
&&\sum_{L\geq 4dC_1m +1}  \max_{\substack{\gamma =(y_i)_{i=1}^{\ell}\subset (x_k-{\rm B}(m+L))\\|\gamma|_1 \leq L}}  \sum \limits_{i=1}^{\ell-1}{\rm T}_1(y_i,y_{i+1}) \I(\kE_{k,L})  \notag \\
& \leq &  \sum_{L\geq 4dC_1m +1} \E  \left(\max_{\substack{\gamma =(y_i)_{i=1}^{\ell}\subset (x_k-{\rm B}(2L))\\|\gamma|_1 \leq 2L}}  \sum \limits_{i=1}^{\ell-1}  {\rm T}_{1}(y_i,y_{i+1}) \I(\kE_{k,L})\right) \notag\\
&\leq &   \sum_{L\geq 4dC_1m +1} \left[\E  \left(\max_{\substack{\gamma =(y_i)_{i=1}^{\ell}\subset (x_k-{\rm B}(2L))\\|\gamma|_1 \leq 2L}}  \sum \limits_{i=1}^{\ell-1}  {\rm T}_1(y_i,y_{i+1}) \right)^2 \right]^{1/2} \pp(\kE_{k,L})^{1/2} \notag\\
&\leq&\sum_{L\geq 4dC_1m +1} \left[\E  \left(\max_{\gamma =(y_i)_{i=1}^{\ell} \in \kP_{2L}}  \sum \limits_{i=1}^{\ell-1}  {\rm T}_1(y_i,y_{i+1}) \right)^2 \right]^{1/2} \pp(\kE_{k,L})^{1/2},
\end{eqnarray}
where we have used the Cauchy-Schwarz inequality in the second inequality.\\

These yield that
\begin{equation}\label{collect}
  \begin{split}
    \E|\Delta_k| &\leq \frac{2}{\#{\rm B}(m)} \E \left(\max_{\gamma =(y_i)_{i=1}^{\ell} \in \kP_{8dC_1m}}  \sum \limits_{i=1}^{\ell-1}  {\rm T}_1(y_i,y_{i+1}) \right)\\
    &\hspace{4mm} +\frac{2}{\#{\rm B}(m)}\sum_{L\geq 4dC_1m +1} \left[\E  \left(\max_{\gamma =(y_i)_{i=1}^{\ell} \in \kP_{2L}}  \sum \limits_{i=1}^{\ell-1}  {\rm T}_1(y_i,y_{i+1}) \right)^2 \right]^{1/2} \pp(\kE_{k,L})^{1/2}.
  \end{split}
  \end{equation}
Using similar arguments for  \eqref{s15}, \eqref{s16} and \eqref{s17}, we can show  that
\begin{eqnarray} \label{s18}
\sum_{k=1}^{\infty}\E(|\Delta_k|) &\leq& \frac{2}{\#{\rm B}(m)} \sum_{k=1}^{\infty}  \sum \limits_{z\in {\rm B}(m)}   \E^{\otimes 2}  \tilde{\rm T}_{x_k}(x_k,\bar{x}_k) \mathbb{I}(x_k \in \gamma_{z,z+x}) \notag\\
&=& 2\E^{\otimes 2} \left(\sum \limits_{y\in \Z^d} \tilde{\rm T}_{y}(y, \bar{y} ) \I(y \in \gamma_{0,x})\right) \notag \\
&\leq& 2 \E \left( \max_{\gamma =(y_i)_{i=1}^{\ell}\in \kP_{C_1|x|_1}}  \sum \limits_{i=1}^{\ell-1} {\rm T}_1(y_i,y_{i+1})\right) \notag\\
&&+ 2\sum_{L\geq C_1|x|_1+1} \left[\E  \left(\max_{\gamma =(y_i)_{i=1}^{\ell} \in \kP_{L}}  \sum \limits_{i=1}^{\ell-1}  {\rm T}_1(y_i,y_{i+1}) \right)^2 \right]^{1/2} \pp(\kE_{L})^{1/2}, 
\end{eqnarray}
where we define
$$\kE_{L}=\{|\gamma_{0,x}|_1=L\}.$$
 \begin{lem} \label{emx}
 	There exists a positive constant $C$ such that for  all  $L \geq 1$,
 	\begin{itemize}
 		\item [(i)] $$ \E \left( \max_{\gamma =(y_i)_{i=1}^{\ell} \in \kP_{L}}  \sum \limits_{i=1}^{\ell-1} {\rm T}_1(y_i,y_{i+1})\right) \leq CL.$$ 	
 		\item[(ii)] $$\E  \left(\max_{\gamma =(y_i)_{i=1}^{\ell} \in \kP_L}  \sum \limits_{i=1}^{\ell-1}  {\rm T}_1(y_i,y_{i+1}) \right)^2 \leq CL^{4}.$$
 	\end{itemize}
 \end{lem}

\vspace{0.3 cm}

 We postpone the proof of this lemma  to Section \ref{sec:lems}.
 
 \vspace{0.3 cm}
 
 \begin{lem}\label{maximal-jump-lem}
   Given a path $\gamma=(y_i)^{\ell}_{i=1}\subset \Z^d$, we define the maximal jump $$\mathcal{M}(\gamma)=\max_{1\leq i\leq \ell-1}|y_i-y_{i+1}|_1.$$
   Then, there exists $\e>0$ independent of $x$ such that for any $L\geq m=|x|_1^{1/4}$,
 $$\P(\mathcal{M}(\gamma_{0,x})\geq L)\leq e^{-L^{\e}}.$$
 \end{lem}
 \begin{proof}
   We write $ \gamma_{0,x}=(y_i)^{\ell}_{i=1}$. If $|y_i-y_{i+1}|_1\geq L $, then ${\rm T}(y_i,y_{i+1})=t(y_i,y_{i+1})\geq L$. By the union bound, Lemma~\ref{l1} and Lemma~\ref{l3}, we have
    \begin{equation} \label{maximal-jump}
      \begin{split}
        &\P(\mathcal{M}(\gamma_{0,x})\geq L) \\
        & \leq \P(\exists u,v\in {\rm B}(C_1|x|_1+L)\text{ s.t. ${\rm T}(u,v)=t(u,v)\geq L$})+\P(|\gamma_{0,x}|_1\geq C_1|x_1|+L)\\
        &\leq  [\#{\rm B}(C_1|x|_1+L)]^2 \max_{u,v\in {\rm B}(C_1|x|_1+L)}\P({\rm T}(u,v)=t(u,v)\geq L)+\P({\rm T}(0,x)\geq C_1|x_1|+L)\\
        &\leq  e^{-L^{\e}},
      \end{split}
    \end{equation}
    for some constant $\e>0$.
\end{proof}
 
 \subsubsection{Proof of Proposition \ref{prop2} (ii)}
Fix $k \geq 1$. We first estimate $\pp(\kE_{k,L})$.   Assume that $\kE_{k,L}$ occurs and $\gamma_{0,x} \cap (x_k-{\rm B}(m)) = (y_i)_{i=1}^{\ell}$. Then 
 $$L=\sum_{y \in \gamma_{0,x} \cap x_k -{\rm B}(m)} |y-\bar{y}|_1 \leq \sum_{i=1}^{\ell -1}  t(y_i,y_{i+1}) + t(y_{\ell},\bar{y}_{\ell}) = {\rm T}(y_1,\bar{y}_{\ell}).$$
 Moreover,   $\bar{y}_{\ell} \in x_k-{\rm B}(m+\mathcal{M}(\gamma_{0,x}))$, since $|y_{\ell}-\bar{y}_{\ell}|_1 \leq\mathcal{M}(\gamma_{0,x}) $ and $y_{\ell} \in x_k-{\rm B}(m)$.
Therefore, using the union bound, Lemma \ref{l1} and Lemma~\ref{maximal-jump-lem}, for $L\geq 4dC_1m+1$,
\begin{eqnarray*}
  \pp( \kE_{k,L}) &\leq &  \pp \left(  \exists u,v \in x_k-{\rm B}(m+\mathcal{M}(\gamma_{0,x}))\text{ such that }{\rm T}(u,v) \geq L  \right) \notag \\
  & \leq  & \pp \left(  \exists u,v \in {\rm B}(m+(L/4dC_1))\text{ such that }{\rm T}(u,v) \geq L  \right)+\P(M(\gamma_{0,x})\geq L/4dC_1) \notag \\
&\leq & (2(m+L))^{2d} e^{-L^{\e}}+ e^{-L^{\e}} \leq  (4(m+L))^{2d} e^{-L^{\e}}. 
\end{eqnarray*}
Combining this inequality with \eqref{collect} and Lemma \ref{emx}, we obtain that there exists $C>0$ such that for any $k\geq 1$
\begin{eqnarray} \label{s19}
  \E(|\Delta_k|) &\leq&  \frac{C}{\#{\rm B}(m)} \left(m + \sum_{L\geq 4d C_1m} L^2 (4(m+L))^d e^{-L^{\e}/2}\right) \notag\\
  &=&\kO(m^{1-d}) = \kO(|x|_1^{(1-d)/4}).
\end{eqnarray}
Since ${\rm T}(x) \geq |\gamma_{0,x}|_1$, by using Lemma \ref{l1}, for any $L \geq C_1|x|_1$
\begin{eqnarray} \label{el}
\pp(\kE_L) \leq \pp({\rm T}(x) \geq L) \leq e^{-L^{\varepsilon_1}}.
\end{eqnarray}
Using this inequality, \eqref{s18} and Lemma \ref{emx}, we get
\begin{eqnarray} \label{s20}
  \sum_{k\geq 1}\E(|\Delta_k|) &\leq& C \left(|x|_1 + \sum_{L\geq  C_1|x|_1} L^2  e^{-L^{\varepsilon}/2}\right) \notag \\
  &=& \kO(|x|_1).
\end{eqnarray}
Now, Proposition \ref{prop2} (ii) follows from  \eqref{s19} and \eqref{s20}, since
$$\sum_{k\geq 1} (\E(|\Delta_k|))^2\leq \left(\max_{k\geq 1}\E|\Delta_k|\right)\left(\sum_{k\geq 1} \E|\Delta_k|\right).$$

\subsubsection{ Proof of Proposition \ref{prop2} (i)} To estimate ${\rm Ent}(\Delta_k)$, we  decompose a simple random walk $({\rm S}^{x_i}_.)$ into the sum of i.i.d. random variables. More precisely, for any $x_i \in \Z^d$ and $j \geq 1$, we write 
$${\rm S}^{x_i}_j=x_i+ \sum_{r=1}^j \w_{i,r},$$
 where $(\w_{i,r})_{i,r \geq 1}$ is an array of i.i.d.\  uniform random variables taking value in the set of canonical coordinates in $\Z^d$, denoted by
 $$\kB_d=\{e_1,\ldots,e_{2d}\}.$$
 Therefore, we can view ${\rm T}(u,v)$ and ${\rm F}_m$ as a function of $(\w_{i,r})$, and hence we sometimes write ${\rm T}(u,v)={\rm T}(u,v,\w)$ to make  the dependence of ${\rm T}(u,v)$ on $\w$ precise.   We define
 $$\Omega=\prod_{i,j\in\N} \Omega_{i,j},$$
 where $\Omega_{i,j}$ is a copy of $\kB_d$. The measure on $\Omega$ is $\pi=\prod_{i,j\in\N} \pi_{i,j}$, where $\pi_{i,j}$  is the uniform measure on $\Omega_{i,j}$. Then we can consider ${\rm F}_m$ as a random variable on the probability space $(\Omega, \pi)$. Given $\w \in \Omega, e \in \kB_d $ and $i,j\in\N$, we define a new configuration $\w^{i,j,e}$ as

 $$\w^{i,j,e}_{k,r}= \begin{cases}
\w_{k,r} \quad \textrm{ if } (k,r) \neq (i,j)\\
e\quad \quad \,\, \textrm{ if } (k,r) = (i,j).
 \end{cases}$$
 
 We define
 \begin{eqnarray}
 \Delta_{i,j}f= \left[{\rm E}\Big(|f(\w^{i,j,U})-f(\w^{i,j,\tilde{U}})|^2\Big)\right]^{1/2},
 \end{eqnarray}
where the expectation  runs over two independent random variables $U$ and $\tilde{U}$, with the same law as the uniform distribution on $\kB_d$.
\begin{lem} \label{del}
	We have 
	$$\sum_{k=1}^{\infty} {\rm Ent}(\Delta_k^2) \leq 2d\sum_{i=1}^{\infty} \sum_{j=1}^{\infty} \E_{\pi} [(\Delta_{i,j}{\rm F}_m)^2].$$
	\end{lem}
\begin{proof}
	We recall that $\Delta_k=\E({\rm F}_m\mid \kF_k)-\E({\rm F}_m\mid \kF_{k-1})$, where 
	$$\kF_k=\sigma (({\rm S}_j^{x_i}), i \leq k, j \geq 1) = \sigma (\w_{i,j}, i\leq k, j \geq 1).$$
Notice  that $\Delta_k^2 \in {\rm L}^2$, since ${\rm T}(x) \in {\rm L}^4$ by Lemma \ref{l1}. Hence, using the tensorization of entropy (Lemma \ref{ent}), we have for $k\geq 1$, 
	$${\rm Ent}(\Delta_k^2)={\rm Ent}_{\pi}(\Delta_k^2) \leq \E_{\pi}\sum_{i=1}^{\infty} \sum_{j=1}^{\infty} {\rm Ent}_{\pi_{i,j}}\Delta_k^2.$$
	By Lemma \ref{bona},
	$${\rm Ent}_{\pi_{i,j}} \Delta_k^2 \leq 2d  (\Delta_{i,j}\Delta_k)^2.$$
	Thus 
	\begin{eqnarray} \label{15}
	\sum_{k=1}^{\infty} {\rm Ent}(\Delta_k^2) \leq 2d\sum_{j=1}^{\infty}\sum_{i=1}^{\infty} \sum_{k=1}^{\infty} \E_{\pi} [(\Delta_{i,j}\Delta_k)^2].
	\end{eqnarray}
     We fix $i,j$.   We define the filtration $\tilde{\mathcal{F}_k}$ as $\tilde{\mathcal{F}_k}=\mathcal{F}_k$ if $k<i$, and $\tilde{\mathcal{F}_k}=\mathcal{F}_k\lor \sigma(U,\tilde{U})$ if $k\geq i$. For simplicity of notation, we denote $\E=\E_\pi {\rm E}$.  Since
         \begin{align*}
           \E_{\pi} [(\Delta_{i,j}\Delta_k)^2]&=\E[(\E[\rmf_m(\w^{i,j,U})-\rmf_m(\w^{i,j,\tilde{U}})|~\tilde{\mathcal{F}}_k]- \E[\rmf_m(\w^{i,j,U})-\rmf_m(\w^{i,j,\tilde{U}})|~\tilde{\mathcal{F}}_{k-1}])^2]\\
           &= \E[(\E[\rmf_m(\w^{i,j,U})-\rmf_m(\w^{i,j,\tilde{U}})|~\tilde{\mathcal{F}}_k])^2] - \E[(\E[\rmf_m(\w^{i,j,U})-\rmf_m(\w^{i,j,\tilde{U}})|~\tilde{\mathcal{F}}_{k-1}])^2],
           \end{align*}
	 $$\E[(\E[\rmf_m(\w^{i,j,U})-\rmf_m(\w^{i,j,\tilde{U}})|~\tilde{\mathcal{F}}_{0}])^2]=0,$$
         and $$\lim_{k\to\infty}\E[(\E[\rmf_m(\w^{i,j,U})-\rmf_m(\w^{i,j,\tilde{U}})|~\tilde{\mathcal{F}}_{k}])^2]=\E_{\pi} [(\Delta_{i,j}{\rm F}_m)^2],$$ we get 
	$$ \sum_{k=1}^{\infty} \E_{\pi} [(\Delta_{i,j}\Delta_k)^2]= \E_{\pi} [(\Delta_{i,j}{\rm F}_m)^2],$$
	for any $i,j$.        Combining this equation with \eqref{15}, we get the desired result.
\end{proof}

\vspace{0.2 cm}
\begin{proof}[Proof of Proposition \ref{prop2} (i).]
	Using  Lemma \ref{del} and Jensen's inequality, we get  
\begin{eqnarray} \label{16'}
\sum_{k=1}^{\infty} {\rm Ent}(\Delta_k^2) &\leq& 2d \sum_{i=1}^{\infty} \sum_{j=1}^{\infty} \E_{\pi} [(\Delta_{i,j}{\rm F}_m)^2] \notag \\ 
&\leq& \frac{2d}{\#{\rm B}(m)} \sum_{z\in {\rm B}(m)} \sum_{i=1}^{\infty} \sum_{j=1}^{\infty} \E_{\pi} [(\Delta_{i,j}{\rm T}(z,z+x))^2].
\end{eqnarray}
By the translation invariance of the passage times, we reach
\begin{eqnarray} \label{16}
\sum_{k=1}^{\infty} {\rm Ent}(\Delta_k^2) \leq 2d  \sum_{i=1}^{\infty} \sum_{j=1}^{\infty} \E_{\pi} [(\Delta_{i,j}{\rm T}(x))^2].
\end{eqnarray}
On the other hand,
\begin{eqnarray*}
&&\E_{\pi}[(\Delta_{i,j}{\rm T}(x))^2]=\E_{\pi} [({\rm E}|{\rm T}(x,\w^{i,j,U})-{\rm T}(x,\w^{i,j,\tilde{U}})|)^2] \notag\\
&\leq&\E_{\pi} {\rm E}[|{\rm T}(x,\w^{i,j,U})- {\rm T}(x,\w^{i,j,\tilde{U}})|^2]\notag\\
&=&\E_{\pi} {\rm E}[|{\rm T}(x,\w^{i,j,U})- {\rm T}(x)|^2]\notag\\
&=& 2\E_{\pi} {\rm E}[({\rm T}(x,\w^{i,j,U})- {\rm T}(x))^2\I({\rm T}(x,\w^{i,j,U}) \geq {\rm T}(x))].
\end{eqnarray*}
We observe that if  $x_i \not \in \gamma_{0,x}$, or $x_i  \in \gamma_{0,x}$ but ${\rm T}(x_i,\bar{x}_i) <j$, then 
$${\rm T}(x,\w^{i,j,U}) \leq {\rm T}(x).$$
 Otherwise, assume that $x_i  \in \gamma_{0,x}$ and ${\rm T}(x_i,\bar{x}_i) \geq j$. Then   for any $e \in \kB_d$, 
\begin{eqnarray*}
{\rm T}(x_i,\bar{x}_i) \geq {\rm T}(x_i,\bar{x}_i+e-\w_{i,j},\w^{i,j,e}),
\end{eqnarray*}
since if we only replace   $\w_{i,j}$ by $e$, by $t(x_i, \bar{x}_i)$ (also equals ${\rm T}(x_i,\bar{x}_i)$, as $x_i \sim \bar{x}_i \in \gamma_{0,x}$) steps,  the simple random walk $({\rm S}^{x_i}_.)$ arrives at $\bar{x}_i+e-\w_{i,j}$. Moreover, 
\begin{eqnarray*}
{\rm T}(x_i, \w^{i,j,e}) = {\rm T}(x_i), \quad  {\rm T}(\bar{x}_i,x,\w^{i,j,e}) \leq {\rm T}(\bar{x}_i,x), 
\end{eqnarray*}
and 
\begin{eqnarray*}
{\rm T}(x)&=&{\rm T}(x_i)+{\rm T}(x_i,\bar{x}_i) + {\rm T}(\bar{x}_i,x) ,\\
{\rm T}(x,\w^{i,j,e})&\leq&{\rm T}(x_i,\w^{i,j,e})+{\rm T}(x_i,\bar{x}_i -e+\w_{i,j},\w^{i,j,e}) \\
&&+{\rm T}(\bar{x}_i -e+\w_{i,j},\bar{x}_i,\w^{i,j,e})  + {\rm T}(\bar{x}_i,x,\w^{i,j,e}). 
\end{eqnarray*}
Therefore, we reach
\begin{eqnarray*}
{\rm T}(x,\w^{i,j,U})-{\rm T}(x) \leq {\rm T}(\bar{x}_i -U+\w_{i,j},\bar{x}_i,\w^{i,j,U}) \leq \max_{y:|y-\bar{x}_i|_1 \leq 2} {\rm T}(y,\bar{x}_i,\w^{i,j,U}).
\end{eqnarray*}
Furthermore, since $\w$ differs from $\w^{i,j,U}$ only in the trajectory of $({\rm S}^{x_i}_.)$,  for any $u,v \in \Z^d$, 
\begin{eqnarray}
{\rm T}(u,v,\w^{i,j,U}) \leq {\rm T}^{[x_i]}(u,v) \leq {\rm T}_2(u,v),
\end{eqnarray}
where we define
\begin{equation}\label{def:t2}
  {\rm T}_2 (u,v)=\sup_{z \in \Z^d}  {\rm T}^{[z]}(u,v).
  \end{equation}
Therefore, we have
\begin{eqnarray*}
\E_{\pi}[(\Delta_{i,j}{\rm T}(x))^2] \leq 2\E \Big[\max_{y:|y-\bar{x}_i|_1 \leq 2} {\rm T}_2(y,\bar{x}_i)^2 \I(x_i\sim \bar{x}_i \in \gamma_{0,x}, {\rm T}(x_i,\bar{x}_i) \geq j) \Big],
\end{eqnarray*}
and thus
\begin{eqnarray*}
&&\sum_{j=1}^{\infty}\E_{\pi}(\Delta_{i,j}{\rm T}(x))^2 \\
 &\leq& 2\E  \Big[\sum_{j=1}^{\infty}\max_{y:|y-\bar{x}_i|_1 \leq 2} {\rm T}_2(y,\bar{x}_i)^2\I(x_i\sim \bar{x}_i \in \gamma_{0,x}, {\rm T}(x_i,\bar{x}_i) \geq j) \Big] \notag\\
 &=&2\E \Big[{\rm T}(x_i,\bar{x}_i)\max_{y:|y-\bar{x}_i|_1 \leq 2} {\rm T}_2(y,\bar{x}_i)^2 \I(x_i\sim \bar{x}_i \in \gamma_{0,x})   \Big]  \notag\\
 &\leq &  \E  \Big[({\rm T}(x_i,\bar{x}_i)^2 +\max_{y:|y-\bar{x}_i|_1 \leq 2} {\rm T}_2(y,\bar{x}_i)^4) \I(x_i\sim \bar{x}_i \in \gamma_{0,x})   \Big]. 
\end{eqnarray*}
This yields that
\begin{eqnarray} \label{s26}
&&\sum_{i=1}^{\infty}\sum_{j=1}^{\infty}\E_{\pi}(\Delta_{i,j}{\rm T}(x))^2 \notag \\ 
&\leq &  \E  \Big[\sum_{i=1}^{\infty}({\rm T}(x_i,\bar{x}_i)^2 +\max_{y:|y-\bar{x}_i|_1 \leq 2} {\rm T}_2(y,\bar{x}_i)^4) \I(x_i\sim \bar{x}_i \in \gamma_{0,x})   \Big] \notag \\
&= &  \E \Big[\sum_{i=1}^{\infty}({\rm T}(x_i,\bar{x}_i)^2 +\max_{y:|y-\bar{x}_i|_1 \leq 2} {\rm T}_2(y,\bar{x}_i)^4) \I(x_i\sim \bar{x}_i \in \gamma_{0,x})   \Big] \notag \\
& \leq  &  \E \left(\sum_{y\in \gamma_{0,x}}{\rm T}(y,\bar{y})^2\right)   +   \E \left(\sum_{y\in \gamma_{0,x}} \max_{u:|u-y|_1 \leq 2}{\rm T}_2(u,y)^4\right).
\end{eqnarray}
Now using the same arguments for \eqref{s17} and \eqref{s18}, we get
\begin{eqnarray} \label{s27}
  && \E \left(\sum_{y\in \gamma_{0,x}} \max_{|u-y|_1 \leq 2}{\rm T}_2(u,y)^4\right)\\
  &\leq& \E \left( \max_{\gamma =(y_i)_{i=1}^{\ell} \in \kP_{C_1|x|_1}}  \sum \limits_{i=1}^{\ell}  \max_{|u-y_i|_1 \leq 2}{\rm T}_2(u,y_i)^4\right) \notag \\ 
 &&+ \sum_{L\geq C_1|x|_1+1} \left[\E  \left(\max_{\gamma =(y_i)_{i=1}^{\ell} \in \kP_L}  \sum \limits_{i=1}^{\ell}  \max_{|u-y_i|_1 \leq 2}{\rm T}_2(u,y_i)^4 \right)^2 \right]^{1/2} \pp(\kE_{L})^{1/2}. 
 \end{eqnarray}
\begin{lem} \label{t2l} As $|x|_1$ tends to infinity,
	$$ \E \left(\sum_{y\in \gamma_{0,x}}{\rm T}(y,\bar{y})^2\right) = \kO(|x|_1).$$
\end{lem} 

\begin{lem} \label{mtl}
	There exists a positive constant $C$ such that for any $L\geq 1$,
	\begin{itemize}
		\item [(i)] 	$$ \E \left( \max_{\gamma =(y_i)_{i=1}^{\ell} \in \kP_L}  \sum \limits_{i=1}^{\ell}  \max_{u:|u-y_i|_1 \leq 2} {\rm T}_2(u,y_i)^4\right) \leq CL.$$ 
		\item[(ii)] $$ \E  \left(\max_{\gamma =(y_i)_{i=1}^{\ell} \in \kP_L}  \sum \limits_{i=1}^{\ell-1}  \max_{u:|u-y_i|_1 \leq 2}{\rm T}_2(u,y_i)^4 \right)^2 \leq CL^{10}. $$
	\end{itemize}
\end{lem} 
We postpone the proofs of the above two lemmas to Section \ref{sec:lems} and  complete the proof of Proposition~\ref{prop2}. Combining \eqref{el}, \eqref{s26}, \eqref{s27} and Lemmas \ref{t2l} and \ref{mtl}, we get
\begin{eqnarray*}
  \sum_{i=1}^{\infty}\sum_{j=1}^{\infty}\E_{\pi}(\Delta_{i,j}{\rm T}(x))^2 &\leq&  C \left(|x|_1 + \sum_{L\geq  C_1|x|_1} L^2  e^{-L^{\varepsilon_1}/2}\right) \\
  &=& \kO(|x|_1).
\end{eqnarray*}
Thus, we can conclude the proof of Proposition \ref{prop2} by \eqref{16}. 
\end{proof}

\subsection{Proof of  Lemmas~\ref{emx}, \ref{t2l} and \ref{mtl}} \label{sec:lems}
Before presenting the proof of these lemmas, we first show the large deviation estimates as in Lemma~\ref{l1} for ${\rm T}_1$ and ${\rm T}_2$.
\begin{lem} \label{lemttt}
The following statements hold.
\begin{itemize}
	\item [(i)] For any $u,v \in \Z^d$ and $n\geq 1$, the events $\{\rmt (u,v) \leq n \}, \{\rmt_1(u,v) \leq n \}$ and $\{\rmt_2(u,v) \leq n \}$ depend only on SRWs $\{(S^x_.) :|x-u|_1 \leq n\}$.
	\item[(ii)] There exist an integer $C_1\geq 1$  and a positive  constant $\varepsilon_1$ such that for  $k\geq C_1 |y|_1$,
	\begin{equation*} \label{epsi}
	\max \{\pp({\rm T}(0,y)\geq k), \pp({\rm T}_1(0,y)\geq k), \pp({\rm T}_2(0,y)\geq k) \} \leq e^{-k^{\e_1}}.
	\end{equation*}
\end{itemize}
\end{lem}
\begin{proof} For any $u \in \Z^d$ and $n\geq 1$, an event $\kA$ is called $\kF^u_n$ -measurable if $\kA$ depends only on the SRWs $\{(S^x_.) :|x-u|_1 \leq n\}$.	 It directly follows from definition of $\rmt$ that the event $\{\rmt(u,v) \leq n\}$ is $\kF^u_n$ -measurable. By definition of $\rmt_1$ as in \eqref{def:t1}, 
\begin{equation}
\{\rmt_1(u, v) \leq n\}=  \bigcap_{z:|z-u|\leq 1} \{\rmt(z,v) \leq n-1\}.
\end{equation}
In addition  the event  $\{\rmt(z,v) \leq n-1\}$ is $\kF^{z}_{n-1}$ -measurable and $\kF^{u_1}_{n-1}\subset \kF^{u}_{n}$ if $|z-u|_1\leq 1$, so the event $\{\rmt_1(u, v) \leq n\}$ is $\kF^{u}_{n}$ -measurable. Moreover, since $\{\rmt^{[z]}(u, v) \leq n\}$ is $\kF^{u}_{n}$-measurable for any $z\in\Z^d$, the event $\{\rmt_2(u,v) \leq n\}=\cap_z \{\rmt^{[z]}(u, v) \leq n\}$ is $\kF^{u}_{n}$ -measurable.  We now prove (ii).\\

 By repeating the arguments of the proof of  Lemma \ref{l1} (see \cite[Proposition 2.4]{K} or \cite[Lemma 4.2]{AMP}), we can show that there exist positive constants $C$ and $\varepsilon$ such that for any $y, z\in\Z^d$, and $t\geq C|y|_1$,
\begin{eqnarray} \label{s39}
\pp\left({\rm T}^{[z]}(0,y) \geq t\right) \leq e^{-t^{\varepsilon}}.
\end{eqnarray}
By the union bound, for $t\ge C_2|y|_1$ with $C_2=2C$, we have
\begin{equation} \label{eps2}
\begin{split}
\P({\rm T}_1(0,y)\geq t)&\leq \sum_{z\in\Z^d:|z|_1=1}\P({\rm T}^{[0]}(z,y)\geq t-1)\\
&\leq 2d e^{-(t-1)^{\e}} \leq e^{-t^{\e_2}}, 
\end{split}
\end{equation}
with some $\e_2>0$, where we have used \eqref{s39} for $t-1\geq 2C|y|_1-1 \geq C|z-y|_1$.\\

We observe also that if ${\rm T}(y) \leq k$ then ${\rm T}^{[z]}(0,y) ={\rm T}(y)$ for $z \not \in {\rm B}(k)$. Therefore, for $k\geq C_3
|y|_1$ with $C_3=\max\{C_1,C_2\}$, 
\begin{eqnarray} \label{eps3}
\pp\left({\rm T}_2(0,y) \geq k\right) &\leq & \pp({\rm T}(y)\geq k) + \pp({\rm T}(y)< k, {\rm T}_2(0,y) \geq k ) \notag\\ 
&\leq &\pp({\rm T}(y)\geq k) + \sum_{z \in {\rm B}(k)} \pp \left({\rm T}^{[z]}(0,y) \geq k\right)  \notag\\
&\leq & e^{-k^{\varepsilon_1}} + (2k+1)^de^{-k^{\varepsilon_2}}\leq  e^{-k^{\e_3}}, \label{s41}
\end{eqnarray}
with some $\e_3>0$. Combining \eqref{eps2} and \eqref{eps3} with Lemma \ref{l1}, we get (ii).	
\end{proof}

\subsubsection{Proof of Lemma \ref{t2l}}
We decompose
\begin{eqnarray*}
	\E \left(\sum_{y\in \gamma_{0,x}}{\rm T}(y,\bar{y})^2\right)&=&\E\left[\sum_{y\in \gamma_{0,x}}{\rm T}(y,\bar{y})^2;~{\rm T}(x)\leq C|x|_1\right]+\E\left[\sum_{y\in \gamma_{0,x}}{\rm T}(y,\bar{y})^2;~{\rm T}(x)> C|x|_1\right].
\end{eqnarray*}
By a similar argument as in Lemma~\ref{l2}, the second term can be bounded from above by
\begin{equation}
\begin{split}
&\left(\E\left[\left(\sum_{y\in \gamma_{0,x}}{\rm T}(y,\bar{y})^2\right)^2\right]\right)^{1/2}\P(|\gamma_{0,x}|_1>C_1|x|_1|)^{1/2}\\
\leq & \left(\E\left[{\rm T}(x)^4\right]\right)^{1/2}\P({\rm T}(x)>C_1|x|_1|)^{1/2}\\
\leq &C|x|^2_1e^{-|x|_1^\e/2},
\end{split}
\end{equation}
and thus for all $|x|_1$  large enough,
\begin{equation} \label{tmc1}
\E \left(\sum_{y\in \gamma_{0,x}}{\rm T}(y,\bar{y})^2\right) \leq \E\left[\sum_{y\in \gamma_{0,x}}{\rm T}(y,\bar{y})^2;~{\rm T}(x)\leq C|x|_1\right]+ 1.
\end{equation}
   For any $\gamma=(y_i)_{i=1}^{\ell}$, we define 
\begin{eqnarray*}
{\rm A}^{\gamma}_{M}&=&\{y_i \in \gamma: {\rm T}(y_i,y_{i+1}) = M\}.
\end{eqnarray*}
Then, we can express 
\begin{eqnarray} \label{32}
\sum_{i=1}^{\ell-1} {\rm T}(y_i,y_{i+1})^2 = \sum_{M \geq 1 } M^2 \#{\rm A}^{\gamma}_M.
\end{eqnarray}
 By definition of $\rmax$,
\begin{eqnarray} \label{amgi}
\#\rmax \I(\rmt (x)\leq C_1|x|_1) &\leq& \I(\gamma_{0,x} \in \kP_{C_1|x|_1}) \sum_{y \in \gamma_{0,x}} \I(T(y,\bar{y})=M)  \notag \\
&=&\I(\gamma_{0,x} \in \kP_{C_1|x|_1}) \sum_{y \in \gamma_{0,x}} \I(T(y,\bar{y})=t(y,\bar{y})=M) \notag \\
& \leq & \I(\gamma_{0,x} \in \kP_{C_1|x|_1}) \sum_{y \in \gamma_{0,x}} I_y,
\end{eqnarray}
where 
\begin{equation}
I_y = \{\exists \, z \in \Z^d: |z-y|_1 \leq M, {\rm T}(y,z)=t(y,z)=M\}.
\end{equation}
By Lemma \ref{lemttt} (i), $\{I_y, y \in \Z^d\}$ is a   collection of $M$-dependent Bernoulli random variables, and thus the condition (E1) in Lemma \ref{wdsp} holds. In addition, it follows from the union bound and Lemma  \ref{l3} that 
\begin{equation} \label{qm}
  \begin{split}
    q_M&= \sup_{y\in \Z^d} \pp(\exists \,z \in \Z^d: |z-y|_1 \leq M, T(y,z)=t(y,z)=M)\\
    &\leq (2M+1)^{d} e^{-M^{\varepsilon}},
    \end{split}
\end{equation}
with $\varepsilon>0$ as in Lemma \ref{l3}. Therefore, the condition (E2) that $q_M \leq (3M+1)^{-d}$ follows if $\exp (M^{\varepsilon}) \geq ((2M+1)(3M+1))^d$, which holds for all $M\geq M_0$, with  $M_0=M_0(d,\varepsilon)$ a large constant.  Now using  \eqref{amgi} and Lemma \ref{wdsp}, we obtain that for $M\geq M_0$,
\begin{eqnarray} \label{am0x}
\E(\#\rmax \I(\rmt (x)\leq C_1|x|_1)) \leq \E \left(\max_{\gamma \in \kP_{C_1|x|_1}} \sum_{y \in \gamma} I_y \right) &\leq& C |x|_1 M^{d+1} q_M^{1/d}  \notag \\
&\leq& C'|x|_1M^{d+2}e^{-M^{\varepsilon}/d}.
\end{eqnarray}
For $M\leq M_0$, it is obvious that
\begin{equation}
\#\rmax \I(\rmt (x)\leq C_1|x|_1) \leq \frac{C_1|x|_1}{M}.
\end{equation}
Combining the last two estimates with \eqref{32}, we arrive at
\begin{eqnarray}
\E\left[\sum_{y\in \gamma_{0,x}}{\rm T}(y,\bar{y})^2;~{\rm T}(x)\leq C|x|_1\right]&=& \E\left[ \sum_{M \geq 1 } M^2 \#{\rm A}^{\gamma_{0,x}}_M;~{\rm T}(x)\leq C|x|_1\right] \notag \\
& \leq & C |x|_1 \left[ \sum_{M=1}^{M_0-1} M + \sum_{M\geq M_0} M^{d+4} \exp(-M^{\varepsilon}/d) \right] = \kO(|x|_1).
\end{eqnarray}
Combining this estimate with \eqref{tmc1}, we get the desired result.

\subsubsection{Proof of Lemma \ref{emx}} We begin with part (ii), which is easier than (i). Observe that 
\begin{eqnarray*}
\max_{\gamma =(y_i)_{i=1}^{\ell} \in \kP_L}  \sum \limits_{i=1}^{\ell-1}  {\rm T}_1(y_i,y_{i+1}) \leq L \max_{u,v \in {\rm B}(L)} {\rm T}_1(u,v)
\end{eqnarray*}
Using the union bound and Lemma \ref{lemttt} (ii),  for any $k\geq 4dC_1 L$,
\begin{eqnarray*}
\pp \left(\max_{u,v \in {\rm B}(L)} {\rm T}_1(u,v) \geq k\right) \leq (2L+1)^{2d} e^{-k^{\varepsilon_1}}.
\end{eqnarray*}
The last two inequalities yield that
\begin{eqnarray*}
  \E \left(\max_{\gamma =(y_i)_{i=1}^{\ell} \in \kP_L}  \sum \limits_{i=1}^{\ell-1}  {\rm T}_1(y_i,y_{i+1}) \right)^2 &\leq& CL^4 \left(1+ (2L+1)^{2d} \sum_{k\geq 4dC_1L} k^2 e^{-k^{\varepsilon_1}}\right)\\
  &=&\kO(L^4).
\end{eqnarray*}
 We now prove (i). For any $\gamma=(y_i)_{i=1}^{\ell} \in \kP_L$, we define
\begin{eqnarray*}
	\bar{A}^{\gamma}_M&=&\{y_i \in \gamma: |y_i-y_{i+1}|_1=M\},\\
	\bar{A}^{\gamma}_{M,0}&=&\{y_i \in \bar{A}^{\gamma}_M: {\rm T}_1(y_i,y_{i+1}) \leq C_1 M\},\\
	\bar{A}^{\gamma}_{M,k}&=&\{y_i \in \bar{A}^{\gamma}_M: {\rm T}_1(y_i,y_{i+1}) = C_1M +k\},
\end{eqnarray*}
with $C_1$ as in Lemma \ref{lemttt} (ii).  
Then
\begin{eqnarray}
\#\bar{A}^{\gamma}_M=\sum_{k\geq 0} \#\bar{A}^{\gamma}_{M,k}, \quad \sum_{M \geq 1 } M \#\bar{A}^{\gamma}_{M} = |\gamma|_1 \leq L.
\end{eqnarray}
Therefore,
\begin{eqnarray} \label{t33}
\sum_{i=1}^{\ell-1} {\rm T}_1(y_i,y_{i+1})& \leq& \sum_{M \geq 1 } \left(  C_1M\#\bar{A}^{\gamma}_{M,0} + \sum_{k \geq 1 }  (C_1M+k) \#\bar{A}^{\gamma}_{M,k}  \right) \notag\\
& \leq & C_1L+ \sum_{M \geq 1 } \sum_{k \geq 1 }  k \#\bar{A}^{\gamma}_{M,k}.
\end{eqnarray}
We shall apply  the same arguments as in the proof of Lemma \ref{t2l} to deal with the sum above. Similarly to \eqref{amgi},
\begin{equation}
\# \bar{A}^{\gamma}_{M,k} \leq \sum_{y \in \gamma} \bar{I}_y, 
\end{equation}
where 
\begin{equation}
 \bar{I}_y=\I\left(\exists z \in \Z^d: |z-y|_1 \leq C_1M+k,  \rmt_1(y,z)=C_1M+k\right).
\end{equation}
By Lemma \ref{lemttt} (i), $\{\bar{I}_y, y\in \Z^d \}$ is a collection of $(C_1M+k)$-dependent Bernoulli random variables. Hence, using the same arguments for \eqref{am0x}, we can prove that for $
C_1M+k\geq M_0$, with $M_0=M_0(d)$ some large constant,
\begin{eqnarray} \label{t34}
\E \left(\max_{\gamma \in \kP_L} \#\bar{A}^{\gamma}_{M,k}\right) \leq CL(C_1M+k)^{d+1} q_{M,k}^{1/d}, 
\end{eqnarray}
where 
\begin{eqnarray*}
q_{M,k}&=&\sup_{y\in \Z^d}\pp\left(\exists z\in \Z^d:  |y-z|_1 \leq C_1M+k, {\rm T}_1(u,v)=C_1M+k \right) \\
&\leq& (2(C_1M+k)+1)^{d}e^{-(C_1M+k)^{\varepsilon_1}},
\end{eqnarray*}
by using the union bound and Lemma \ref{lemttt} (ii). It is obvious that   $\#\bar{A}^{\gamma}_{M,k} \leq |\gamma|_1/(C_1M+k)$ for all $M, k$. Hence,
\begin{equation}
\sum_{M,k: C_1M +k \leq M_0} k\#\bar{A}^{\gamma}_{M,k} \leq M_0^2 |\gamma|_1.
\end{equation}
 Combining \eqref{t33} and \eqref{t34}, we have
\begin{eqnarray*}
  \E \left( \max_{\gamma =(y_i)_{i=1}^{\ell} \in \kP_L}  \sum \limits_{i=1}^{\ell-1} {\rm T}_1(y_i,y_{i+1})\right) &\leq& CL  \left(1+ \sum_{M, k: C_1M+k \geq M_0} (C_1M+k)^{d+3} e^{-(C_1M+k)^{\varepsilon_1}/d}\right)\\
  &=&  \kO(L),
\end{eqnarray*}
for some $C=C(d,M_0)$, which proves (i).

\subsubsection{Proof of Lemma \ref{mtl}}
To show (ii), we notice that 
\begin{equation} \label{td4}
\max_{\gamma =(y_i)_{i=1}^{\ell} \in \kP_L}  \sum \limits_{i=1}^{\ell-1}  \max_{u:|u-y_i|_1 \leq 2}{\rm T}_2(u,y_i)^4 \leq L \max_{u,v \in {\rm B}(L+2)} {\rm T}_2(u,v)^4.
\end{equation}
Now part (ii) follows from   \eqref{epsi} and \eqref{td4} by using   the same arguments as in Lemma \ref{emx} (ii). \\

  The proof of (i) is similar to that of Lemma \ref{t2l}.  As in Lemma \ref{t2l},  we define for $\gamma \in \kP_L$, and $M \geq 1$,
   \begin{eqnarray*}
	A'^{\gamma}_{M}&=& \#\{y \in \gamma: \max_{u:|u-y|_1 \leq 2} {\rm T}_2(u,y)^4 = M\} = \sum_{y \in \gamma}I'_y, 
\end{eqnarray*}
where 
\begin{equation}
I'_y= \I\left( \max_{u:|u-y|_1 \leq 2} {\rm T}_2(u,y)^4 = M\right).
\end{equation}
By Lemma \ref{lemttt} (i), for $M\geq 16$, $\{I'_y, y\in \Z^d \}$ is a collection of $M$-dependent Bernoulli random variables. By Lemma \ref{lemttt} (ii)  and the union bound,
\begin{eqnarray}
q'_M=\sup_{y\in \Z^d}\pp \left( \max_{u:|u-y|_1 \leq 2} {\rm T}_2(u,y)^4 = M \right) &\leq&  e^{-M^{\varepsilon}},
\end{eqnarray}
for some $\varepsilon>0$ small. Repeating the  arguments as in the proof of Lemma \ref{t2l} with  $A'^{\gamma}_{M}, q'_M$ instead of ${\rm A}^{\gamma}_{M}, q_M$, we can show that 
\begin{eqnarray*}
 \E \left( \max_{\gamma =(y_i)_{i=1}^{\ell} \in \kP_L }  \sum \limits_{i=1}^{\ell}  \max_{u:|u-y_i|_1 \leq 2}{\rm T}_2(u,y_i)^4\right) =  
 \kO(L) \left(M_0^2+\sum_{M\geq M_0} M^{d+2} e^{-M^{\varepsilon}/d}\right) =\kO(L),
\end{eqnarray*}
with  $M_0=M_0(d)$ a large constant, which proves (i).

\section{Proof of Proposition \ref{loop} }
\begin{proof}
The upper bound on the length of optimal paths is a consequence of Lemma \ref{l1}. Indeed,	 if $\gamma \in \O(x)$, then $l(\gamma) \leq {\rm T}(x)$. Hence,	by Lemma \ref{l1},
	\begin{eqnarray}
	  \pp \left( \max_{\gamma \in \O(x)}  l(\gamma) > C_1|x|_1  \right) &\leq& \P( {\rm T}(x)> C_1|x|_1)\\
   &\leq&       e^{-|x|_1^{\varepsilon_1}},
	\end{eqnarray}
	with $\varepsilon_1$ and $C_1$  positive constants as in Lemma \ref{l1}. We start the proof of  the lower bound by recalling a definition in the proof of Lemma \ref{t2l}. Given a path $\gamma=(y_i)^{\ell}_{i=0}$,  define 
	$${\rm A}^{\gamma}_{M}=\{0\leq i\leq \ell-1 : {\rm T}(y_i,y_{i+1}) =t(y_i,y_{i+1})= M\}.$$
	Note that $l(\gamma) \geq \sum_{M \geq 1 } \# {\rm A}_M^{\gamma}$ for any $\gamma$. Thus, for any $\gamma \in \O(x)$ and $K\geq 1$
	\begin{eqnarray}
	|x|_1 \leq {\rm T}(x) = \sum_{M \geq  1} M \#{\rm A}^{\gamma}_{M} &\leq& K \sum_{M=1}^K \#{\rm A}^{\gamma}_{M} +\sum_{M\geq K } M \#{\rm A}^{\gamma}_{M} \notag\\
	& \leq & K l(\gamma) + \sum_{M\geq K } M \#{\rm A}^{\gamma}_{M}.
	\end{eqnarray}
	Rearranging it, we obtain that for any $K\geq 1$,
	\begin{eqnarray} \label{t38}
	  \min_{\gamma \in \O(x)} l(\gamma) &\geq& \frac{1}{K} \left(|x|_1- \max_{\gamma  \in \O(x)} \sum_{M\geq  K} M \#{\rm A}^{\gamma}_{M} \right)\notag\\
          &\geq &\frac{1}{K} \left(|x|_1- \sum_{M\geq  K} M \max_{\gamma  \in \O(x)} \#{\rm A}^{\gamma}_{M} \right).
	\end{eqnarray}
        Note that if ${\rm T}(x)\leq C_1|x|_1$, then $\gamma\in \kP_{C_1|x|_1}$ for any $\gamma\in\O(x)$, and thus 
       \begin{equation} \label{ae6}
        \sum_{M\geq  K} M \max_{\gamma  \in \O(x)} \#{\rm A}^{\gamma}_{M}\leq \sum_{M\geq  K} M \max_{\gamma  \in \kP_{C_1|x|_1}} \#{\rm A}^{\gamma}_{M}.
        \end{equation}
	We  define 
	$$M_x=[|x|_1^{1/2(d+3)}],$$
	and 
	$$\kE = \{  \forall \, M\geq M_x, \, \forall \gamma  \in \kP_{C_1|x|_1},~\#{\rm A}^{\gamma}_{M} =0 \}.$$
	Then, by using the union bound and Lemma \ref{l3}, we get
	\begin{eqnarray} \label{ae7}
 \pp(\kE^c)&\leq & \pp (\exists u,v\in {\rm B}(C_1|x|_1)\text{ such that }{\rm T}(u,v)=t(u,v)\geq M_x) \notag\\
		&\leq&   (2C_1|x|_1+1)^{2d} \sum_{M \geq M_x} e^{-M^{\varepsilon_1}}\leq Ce^{-|x|_1^{\e}},  
	\end{eqnarray}
for some positive constants $C$ and $\e$.
  We recall from  the proof of Lemma \ref{t2l} that 
  \begin{equation}
\#\rma_M^{\gamma} \leq \sum_{y\in \gamma} I_y,
\end{equation}
where $\{I_y, y\in \Z^d\}$ is a collection of $M$-dependent Bernoulli random variables 
$$I_y= \I(\exists \, z \in \Z^d: |z-y|_1 \leq M, {\rm T}(y,z)=t(y,z)=M),$$
and 	 
\begin{equation}
q_M=\sup_{y\in\Z^d} \E(I_y) \leq (2M+1)^{d} e^{-M^{\varepsilon}}.
\end{equation} 
Then, the conditions (E1) and (E2) of Lemma \ref{wdsp} are satisfied. Using Lemma \ref{wdsp} (i), we obtain that  
\begin{eqnarray} \label{t41}
\pp \left(\max_{\gamma \in \kP_{C_1|x|_1}} \#{\rm A}^{\gamma}_{M} \geq |x|_1M^{-3}\right) \leq 2^d \exp \left(-|x|_1/((16M)^{d+3})\right),
\end{eqnarray}
provided that $|x|_1M^{-3} \geq C M^{d} \max \{1, |x|_1 Mq_M^{1/d} \}$, which holds for $|x|_1\geq 2CM^{d+5}$ and $M\geq K$ with $K$ a large constant.  
By \eqref{t41}  and the fact that $M_x =[ |x|_1^{1/2(d+3)}] =o( |x|_1^{1/(d+5)})$, 
	\begin{eqnarray*}
	  \pp \left(\sum_{M=K}^{M_x}M \max_{\gamma \in \kP_{C_1|x|_1}}\#{\rm A}^{\gamma}_{M} \geq |x|_1\sum_{M=K}^{M_x}M^{-2}\right) &\leq&  \sum_{M=K}^{M_x} \pp \left(M \max_{\gamma \in \kP_{C_1|x|_1}}\#{\rm A}^{\gamma}_{M} \geq |x|_1 M^{-2}\right)\\
          &\leq & 2^d \sum_{M=K}^{M_x}\exp \left(-\tfrac{|x|_1}{(16M)^{d+3}}\right).
	\end{eqnarray*}
	Therefore,
		\begin{eqnarray} \label{t45}
	\pp \left(\sum_{M=K}^{M_x}M \max_{\gamma \in \kP_{C_1|x|_1}}\#{\rm A}^{\gamma}_{M} >\frac{ |x|_1}{2}\right)\leq e^{-|x|_1^{\varepsilon}},
	\end{eqnarray}
for some $\varepsilon >0$. 	 Combining \eqref{t38}, \eqref{ae6}, \eqref{ae7} and \eqref{t45} yields that 
\begin{equation*}
  \begin{split}
    \pp \left(\min_{\gamma \in \O(x)} l(\gamma)  < \frac{|x|_1}{2K} \right)& \leq \P({\rm T}(x)>C_1|x|_1)+\P(\kE^c)+\P\left(\sum_{M=K}^{M_x}M \max_{\gamma \in \kP_{C_1|x|_1}}\#{\rm A}^{\gamma}_{M}> \frac{|x|_1}{2}\right)\\
    & \leq Ce^{-|x|_1^{\e}},
  \end{split}
 \end{equation*}
which completes the proof of Proposition \ref{loop}.
\end{proof}

\begin{ack} \emph{We would like to thank the anonymous referees for  carefully reading the manuscript and many valuable comments. The work of V. H. Can is supported by  the fellowship of the Japan Society for the Promotion of Science and the Grant-in-Aid for JSPS fellows Number 17F17319, and by the  Vietnam National Foundation for Science and Technology Development (NAFOSTED) under grant 101.03-2017.07.}
\end{ack}

\end{document}